\documentclass[12pt, reqno]{amsart}

\usepackage{amssymb, amsmath, amsthm,stmaryrd}
\usepackage[backref]{hyperref}
\usepackage[backrefs,lite]{amsrefs}
\usepackage{amscd}   
\usepackage{fullpage}
\usepackage{url}
\usepackage[utf8]{inputenc}
\usepackage{hyperref}
\usepackage{lscape}
\usepackage{pdflscape} 
\usepackage[all]{xy} 
\usepackage{MnSymbol,graphicx}
\usepackage{mathtools}
\usepackage{mathrsfs,enumitem}
\usepackage{multirow,qtree}
\DeclareFontEncoding{OT2}{}{} 

\usepackage[usenames,dvipsnames]{color}

\newcommand{\denabs}{b_{\mathrm{abs}}}
\newcommand{\denKWabs}{a_{\mathrm{abs}}}
\newcommand{\denDeltaabs}{\Delta_{\mathrm{abs}}}

\newcommand{\identitymatrix}[1]{\mathrm{Id}_{#1\times #1}}

\usepackage{tikz}

\newcommand{\exponentofinvariant}{\ell}
\newtheorem{lemma}{Lemma}[section]
\newtheorem{theorem}[lemma]{Theorem}
\newtheorem{proposition}[lemma]{Proposition}

\newtheorem{cor}[lemma]{Corollary}

\newtheorem{conjecture}[lemma]{Conjecture}
\newtheorem{claim*}{Claim}

\theoremstyle{definition}
\newtheorem{remark}[lemma]{Remark}

\newtheorem{example}[lemma]{Example}

\def\O{\mathcal{O}}

\newcommand{\PP}{{\mathbb P}}
\newcommand{\C}{{\mathbb C}}

\newcommand{\Q}{{\mathbb Q}}
\newcommand{\R}{{\mathbb R}}
\newcommand{\Z}{{\mathbb Z}}

\newcommand{\Qbar}{{\overline{\Q}}}

\newcommand{\Lbar}{{\overline{L}}}

\newcommand{\Cbar}{{\overline{C}}}

\newcommand{\Jbar}{{\overline{J}}}

\newcommand{\Jcal}{{\mathcal{J}}}

\newcommand{\calC}{{\mathcal C}}

\newcommand{\OO}{{\mathcal O}}

\newcommand{\frakp}{{\mathfrak p}}

\newcommand{\den}{\text{den}}

\DeclareMathOperator{\tr}{tr}

\DeclareMathOperator{\End}{End}

\DeclareMathOperator{\Hom}{Hom}

\DeclareMathOperator{\Aut}{Aut}

\DeclareMathOperator{\Jac}{Jac}

\DeclareMathOperator{\Mat}{M}

\newcommand{\homEA}[1]{\fbox{\raisebox{0ex}[3ex][2ex]{$#1$}}}
\newcommand{\homAE}[1]{\framebox[3em][c]{$#1$}}
\newcommand{\homAA}[1]{\framebox[3em][c]{\raisebox{0ex}[3ex][2ex]{$#1$}}}
\newcommand{\EndEA}[4]{{\setlength\arraycolsep{0pt} \left(\begin{array}{cc} #1 & \homAE{#2} \\ \homEA{#3} & \homAA{#4}\end{array}\right)}}

\newcommand{\homEEEtoEA}[6]{{\setlength\arraycolsep{0pt} \left(\begin{array}{ccc}
			#1 & #2 & #3 \\
			\homEA{#4} &\homEA{#5} &\homEA{#6} \end{array}\right)}}
\newcommand{\homEAtoEEE}[6]{{\setlength\arraycolsep{0pt} 
\left(\begin{array}{ccc}
			#1 & \homAE{#2} \\
			#3 & \homAE{#4} \\
			#5 & \homAE{#6} \end{array}\right)}}

\newcommand{\xyzw}{\EndEA{x}{y}{z}{w}}

\numberwithin{equation}{section}
\numberwithin{table}{section}

\title[Primes dividing invariants of CM Picard curves]{Primes dividing invariants of CM Picard curves}

\author{P{\i}nar K{\i}l{\i}\c{c}er, 
 Elisa Lorenzo Garc\'ia,  Marco Streng}

\thanks{The work of Lorenzo Garc\'ia was partially supported by a project PEPS-Jeunes Chercheur-e-s - 2017.
The work of K{\i}l{\i}\c{c}er was partially supported by DFG priority
project SPP 1489}

 \makeatletter
\let\@wraptoccontribs\wraptoccontribs
\makeatother

\thanks{}

\newcommand{\Addresses}{{
		\bigskip
		\footnotesize
		
		P{\i}nar K{\i}l{\i}\c{c}er, \textsc{ 
			Johann Bernoulli Instituut voor Wiskunde en Informatica, 
			Rijksuniversiteit Groningen,
			Nijenborgh 9,
            9747 AG Groningen, Nederland}\par\nopagebreak
		\textit{E-mail address}, P{\i}nar K{\i}l{\i}\c{c}er: \texttt{p.kilicer@rug.nl}
		
		\medskip
		
		Elisa Lorenzo Garc\'ia, \textsc{IRMAR, Université de Rennes 1,
			Campus de Beaulieu, 
			35042 Rennes Cedex, France}\par\nopagebreak
		\textit{E-mail address}, Elisa Lorenzo Garc\'ia: \texttt{elisa.lorenzogarcia@univ-rennes1.fr}

		\medskip
		
		Marco Streng, \textsc{Mathematisch Instituut,
			Universiteit Leiden,
			P.O. box 9512,
			2300 RA Leiden,
			The Netherlands}\par\nopagebreak \textit{E-mail address}, Marco Streng: \texttt{streng@math.leidenuniv.nl}
		
	}}

\begin{document}

	\begin{abstract}
		We give a bound on the primes dividing the denominators
		of invariants of Picard curves of genus $3$ with complex multiplication.

		Unlike earlier bounds in genus $2$ and~$3$,
		our bound is based not on bad reduction of curves,
		but on a very explicit type of good reduction.
		This approach simultaneously yields a simplification of the proof,
		and much sharper bounds.
				In fact, 
		unlike all previous bounds for genus $3$, our bound is sharp enough for
		use in explicit constructions of Picard curves.
	\end{abstract}
	\maketitle
	

\section{Introduction}

The \emph{Hilbert class polynomial} of an imaginary quadratic field $K$
is the polynomial
whose roots are the $j$-invariants of the elliptic curves $E$
with endomorphism ring isomorphic to the maximal order~$\mathcal{O}_K$.
Its roots generate the Hilbert class field
and
are used for constructing elliptic curves with prescribed
order, which are used in cryptography.

The Hilbert class polynomial has integer
coefficients,
so in order to compute it, it suffices to numerically
approximate its coefficients up to the decimal point.
These ideas can be generalized to curves of genus $g$
as long as their Jacobians have \emph{complex multiplication (CM)}:
The imaginary quadratic field needs to be replaced by
a \emph{CM field} $K$ of degree $2g$, that is, a totally imaginary quadratic 
extension of a totally real number field of degree~$g$,
and we consider curves whose Jacobian
has endomorphism ring isomorphic to~$\mathcal{O}_K$.

Using suitable invariants for 
curves of genus~$g$,
this gives
rise to \emph{class polynomials},
whose coefficients are rational, but not necessarily
integral.
Computational methods for numerically approximating
these polynomials are known for $g\leq 3$
\cites{
	KLLRSS, 			
	KoikeWeng, 			
	LarioSomoza,		
	Wenghypellg3,		
	BILV,				
	BBEL,				
	Spallek,			
	vanWamelen,			
	EisentragerLauter,  
	AroraEisentraeger,  
	GHKRW				
}.

For the efficiency of the methods, as well as a proof of their output, and
a theoretical understanding of the types of $S$-integers created
with such constructions, we need to know the denominators
of the coefficients of these polynomials.

For the case $g=2$, these denominators are now understood,
thanks to the work of Bruinier, Goren, Lauter, Viray and Yang
\cite{GorenLauter, GorenLauter2, LauterViray, BruinierYang}.
The denominators in that case are effectively
computable products of powers of small
primes,
which have been used for computing and proving correctness
of CM curves of genus two~\cite{BouyerStreng}.

In general for $g=3$, it is expected (see \cite[Section~4]{KLLRSS})
that such a result does not hold. However, for the specific
case of hyperelliptic curves ($y^2 = x^8+\cdots$),
a bound on the primes dividing the denominators of invariants is given
by K{\i}l{\i}{\c{c}}er, Lauter, Lorenzo Garc\'ia, Newton, Ozman and
Streng~\cite{KLLNOS}. The proof of~\cite{KLLNOS}
also works for Picard curves ($y^3 = x^4 + \cdots$)
except for a technical conjecture
in~\cite{KLLNOS}, which is work in progress of Reynald Lercier, Qing Liu,
Elisa Lorenzo Garc\'ia and Christophe Ritzenthaler.
The bound in~\cite{KLLNOS} is unfortunately too large to be practical.

All bounds mentioned so far are based on the fact that primes
dividing the denominators of invariants are primes of bad reduction.
The difficulty that makes the bound of \cite{KLLNOS} so large is
that just knowing bad reduction of the curve does not give much information
on the endomorphism structure of the reduction of the Jacobian.

In this paper we propose an alternative set of invariants for Picard curves.
With that set of invariants, the primes dividing the denominators
are primes of a certain very explicit type of reduction
(Lemma~\ref{lem:cases}).

As in \cite{GorenLauter} and \cite{KLLNOS}, we use the embedding
of $\mathcal{O}_K$ into the endomorphism ring of the reduction of the
Jacobian ($\Jac(C)$ modulo $\mathfrak{p}$) 
to get a contradiction for large~$p$.
We get matrices whose entries are quaternion algebra elements.
In loc.~cit.~these quaternion algebra elements are forced to commute
when $p$ gets very large.
Our explicit type of reduction allows us to prove commutativity directly,
so we do not need to assume that $p$ is very large in our proofs.

This drastically reduces the final bound on~$p$
(Theorem~\ref{thm:main}).
In fact, we get a formula for a small set $S$ of small primes
such that the denominators are $S$-units (see Theorem~\ref{thm:main2}).
We also conjecture a bound on the exponents, yielding
a formula for a denominator that is small enough for practical
computation (see Conjecture~\ref{conj:valuations}).
We made a SageMath implementation of our bounds available
online at~\cite{ourcode} and give
numerical examples in Section~\ref{sec:examples}.

\subsubsection*{Acknowledgements}

The authors would like to thank Irene Bouw, Peter Bruin,
Christophe Ritzenthaler, Matthieu Romagny and Anna Somoza
for helpful discussions, 
the anonymous referees for helpful criticism of the exposition in an earlier version,
and Bas Edixhoven for providing the proof
in Appendix~\ref{appendix}.
Most of K{\i}l{\i}\c{c}er's work was done during
her
stay in Universiteit Leiden and Carl von Ossietzky Universit\"at Oldenburg.

	\section{Picard curve, invariants, and statement of the main theorem}  
	
	In this section we introduce Picard curves, a new set of invariants of Picard
	curves,
	and the basic concepts of complex multiplication we will use.
	We also state our main theorem bounding the primes dividing
	the denominators of the invariants.
	
	\subsection{Picard curves}\label{ssec:picard}
	
	Let $L$ be a field of characteristic not~$2$ or~$3$.
	A \emph{Picard curve} of genus~$3$ over~$L$
	is a smooth, projective, plane 
	curve given by an equation of the form 
	\begin{equation}\label{eq:picard}
	C:\,y^3=x^4+ax^2+bx+c,
	\end{equation}
    where $a$, $b$, $c\in L$.
	Suppose that $L$ contains a primitive $3$rd root of unity~$\zeta_3$.
	The automorphism group of $C$ then contains
	$\rho:(x,y)\mapsto(x,\zeta_3y)$ of order $3$.
	The push-forward $\rho_*$ of this automorphism
	is an automorphism of the Jacobian
	$J = \Jac(C)$, and hence is a primitive third root
	of unity in the endomorphism ring.
	By abuse of notation, we denote $\rho_*\in \End(J)$ also by~$\zeta_3$.
	
	Two Picard curves $C : y^3 = x^4 + ax^2+bx+c$ and
	$C' : y^3 = x^4 + a'x^2+b'x+c'$ over a field
	$L$ of characteristic $\nmid 6$
	are isomorphic over $L$ if and only
	if there exists a $\lambda\in L^*$ with
	\begin{equation}\label{eq:isomorphic}
	\lambda^6 a = a',\qquad \lambda^9 b = b',\qquad\mbox{and}\qquad \lambda^{12} c = c'.
	\end{equation}
    
    \subsection{Complex multiplication}\label{ssec:CM}
    
    Let $C$ be a Picard curve.
	We say that $C$ or its Jacobian $J=\Jac(C)$
	has \emph{complex multiplication (CM)}
    if there is a number field $K$ of degree $6$
    and an embedding $\iota:K \rightarrow \End(J)\otimes \Q$.
    From now on, assume that this is the case,
    and let $K$ and $\iota$ be as above.
    We say that $C$ and $J$ have
    \emph{CM by the order $\mathcal{O} = \iota^{-1}(\End(J))\subset K$}.
    
	We say that $J$ has \emph{primitive CM}
	if the embedding $\iota$ gives an isomorphism
	from $K$ to the endomorphism algebra $\End(J_{\overline{L}})\otimes \Q$
	of $J$ over the algebraic closure of~$L$.
	If $\mathrm{char}(L)=0$ and $J$ has CM, then after extending $L$ the
	induced representation on the tangent space is isomorphic to a
	product of embeddings of $K$ into $L$.
	The set $\Phi$ of such embeddings is called the \emph{CM type} of the map $K\rightarrow \End(J)\otimes \Q$.
	It is known that if $J$ has primitive CM if and only
	if $K$ is a CM field and the CM type is \emph{primitive},
	that is, if and only if for every CM subfield $K_1\subsetneq K$,
	the restriction
	of $\Phi$ to $K_1$ is not a CM type
	(see \cite[1.3.5]{Lang83}.)		
	
	Note that if $J$ has primitive CM by $\mathcal{O}$, then 
	the endomorphism $\zeta_3$ of Section~\ref{ssec:picard}
	is (via~$\iota$) a primitive third root of unity in~$\mathcal{O}$.
	Conversely, if a smooth curve of genus~$3$
	has primitive CM by an order
	containing a primitive $3$rd root of unity,
	then the curve is a Picard curve
	(this is stated in the case of maximal orders
	as \cite[Lemma~1]{KoikeWeng},
	though the final sentence of the proof 
	is wrong
	and needs to be replaced by
	\cite[Lemma~7.3 and the paragraph above it]{Hol};
	the proof does not use that the order is maximal, only
	that the third root of unity is in the order).
	
	This allowed Koike and Weng~\cite{KoikeWeng} to
	construct genus-three Picard curves
	with CM by orders in fields of the form $F(\zeta_3)$ for a totally 
    real cubic field~$F$ and a third root of unity~$\zeta_3$,
	just as one can construct elliptic curves with CM by orders
	in imaginary quadratic fields.

	\subsection{Invariants of Picard curves}\label{sec:invariants}
	
	A \emph{homogeneous Picard curve invariant}
	is a weighted homogeneous polynomial $I\in \Z[a,b,c]$
	where $a$, $b$, $c$ are formal polynomial variables of weights $2$, $3$, $4$, respectively.
	For a model as in \eqref{eq:picard} of a Picard curve $C$
	over a field~$L$, we get the value $I(C)\in L$
	by evaluating $I$ in the coefficients of~\eqref{eq:picard}.
	
	For example, we have the
	invariant 
	\begin{equation}\label{eq: disc. inv.}
	\Delta = -4 a^{3} b^{2} + 16 a^{4} c - 27 b^{4} + 144 a b^{2} c - 128 a^{2} c^{2} + 256 c^{3}
	\end{equation}
	of weight~$12$, and $\Delta(C)$ is non-zero for all Picard curves $C$
	as it is the discriminant of the right hand side of~\eqref{eq:picard}.
	An \emph{absolute Picard curve invariant} 
	is a quotient $j = u/b^\exponentofinvariant$ where $u\in\Z[a,b,c]$ has weight~$3\exponentofinvariant$. 
	For example, the three rational functions $j_1 = a^3/b^2$, $j_2 = ac/b^2$,
	and $j_3 = c^3/b^4 = j_1^{-1}j_2^3$
	are absolute Picard curve invariants.
	
	All Picard curves $C$ with $a(C)\not=0$ can 
	be reconstructed up to twists from the values $j_1(C)$
	and $j_2(C)$ of the
	invariants $j_1$ and $j_2$
	as follows.
	Given a Picard curve $C$ over a field~$L$, the curve
	$D: y^3 = x^4 + j_1(C)x^2 + j_1(C)x + j_1(C)j_2(C)$
	is isomorphic to $C$ over the algebraic closure~$\overline{L}$.
	
	Moreover, the three invariants $j_1$, $j_2$, and $j_3$
	generate the ring of all
	absolute Picard curve invariants.
	Indeed, an absolute Picard curve invariant
	is a linear combination of monomials $a^Ac^C/b^B$ with $2A+4C=3B$,
	and each such monomial is a non-negative power of~$j_2$ times a monomial
	with $A=0$ or $C=0$, which in turn is a power of $j_3$ or~$j_1$.

	Instead of the quotients $b^2/a^3=1/j_1$ and $c/a^2=j_2/j_1$
	used by Koike-Weng~\cite{KoikeWeng}, 
	we consider 
	$j_1$, $j_2$, $j_3$ because 
	the primes dividing the denominators of our invariants 
	have nice properties 
	that we can use to find good bounds for them, see Lemma~\ref{lem:cases}
	and Proposition~\ref{prop:decomp}.
	
	\begin{remark}\label{rmk:nonpole}
		If $C$ is a Picard curve of genus $3$ over a number field $L$ with primitive CM,
		and $j=u/b^l$ is an absolute Picard curve invariant,
		then
		$b(C)\neq0$ and $j(C)\in L$.
		This is because if $b(C)=0$, then the curve $C$ admits a non-constant morphism to an
		elliptic curve (formula in \eqref{eq:phi}) and hence its Jacobian is not simple,
		which gives a contradiction with having a primitive CM-type.
	\end{remark}
	
	In particular, for every sextic CM order $\mathcal{O}$ with $\zeta_3\in \mathcal{O}$,
	the class polynomials
	\begin{align}
	 H_{\mathcal{O}, 1}(X) &= \prod_C (X-j_1(C)),\label{eq:classpol1}\\
	 \widehat{H}_{\mathcal{O}, 2}(X) &= \sum_{C} j_2(C)\prod_{D\not\cong C} (X-j_1(D)),\label{eq:classpol2}
	\end{align}
	(with sums and products ranging over the isomorphism classes
	of curves over $\C$ with primitive CM by~$\mathcal{O}$; see~\cite{GHKRW})
	are well defined.

\subsection{Statement of the main result and overview of the proof}

	A weak version of our main theorem is as follows.
	\begin{theorem}\label{thm:main}
		Let $C$ be a Picard curve of genus $3$ 
		over a number field $L$ with $\End(\Jac(C)_{\bar{L}})$
		isomorphic to an order $\O$ of a number field $K$ of degree~$6$. Let $K_+$
		be the real cubic subfield of $K$.
		Let $\mu\in \Z+2\O$ be 
		such that 
		$K_+ = \Q(\mu)$.
		
		Let $j = u/b^\exponentofinvariant$ be an absolute Picard curve invariant.
		Let $\mathfrak{p}$ be a prime of $L$ lying over a rational prime~$p$.
		If $\mathrm{ord}_{\mathfrak{p}}(j(C)) < 0$, then
		$$p \leq \tr_{K_+/\Q}(\mu^2)^3\qquad\mbox{and}\qquad
		p\leq \left(1+\frac{16}{\pi}|\Delta(\O_+)|^{1/2}\right)^{3}< 196|\Delta(\O_+)|^{3/2}.$$
	\end{theorem}

   In Sections~\ref{sec:reduction}--\ref{sec:geometryofnumbers} we prove
   Theorem~\ref{thm:main}.
   We give a stronger version in Section~\ref{sec:setofprimes}.
	The stronger version gives an algorithm for computing a set of primes, instead of just a bound
	on the primes.
	In Section~\ref{sec:conjecture} we also give a conjecture about the powers to which
	such primes appear in the denominators of the invariants.
	A SageMath implementation is available online at
	\cite{ourcode}.
	In Section~\ref{sec:examples} we give examples
	that show that the resulting denominator bounds are small enough
    for practical class polynomial computations.
	
   The first step of the proof of Theorem~\ref{thm:main} is the explicit type of
   reduction that is implied by the appearance of a prime $\mathfrak{p}$ in the invariant~$b$.
   This type of reduction is given in Lemma~\ref{lem:cases}.
   Proposition~\ref{prop:decomp} then shows how this type of reduction
   makes the reduction of the Jacobian decompose into a product
   of an elliptic curve $A_1$ and a principally polarized abelian surface~$A_2$.
   The rest of Section~\ref{sec:reduction} is the proof of Proposition~\ref{prop:decomp}.
   
   Once we know this decomposition of the
   reduction $\overline{J}$ of the Jacobian $J=\Jac(C)$,
   the endomorphisms of $J$ give rise to matrices
   consisting of homomorphisms between the components $A_1$ and $A_2$ of~$\overline{J}$.
   This will make $A_2$ decompose further and give our endomorphisms as matrices
   over the endomorphism ring $\mathcal{R}$ of~$A_1$ (see Section~\ref{sec:decomp}).
   
   An important part of the earlier proofs in genus $2$ and~$3$
   (\cite{GorenLauter, KLLNOS})
   is to force these matrices
   to have entries in a \emph{field} instead of in the quaternion ring~$\mathcal{R}$.
   This was always done by an argument from \cite{GorenLauter}, which uses the fact that
   elements of small norm of quaternion algebras of large discriminant commute.
   In order to be able to use this fact, the prime $p$ needs to be very large, which
   is why bounds based on that type of argument typically are very large
   (an exception is \cite{LauterViray}, which is more
   complicated and has not been generalized
   to genus three yet).
   In Section~\ref{sec:commute} we use the explicit endomorphism $\zeta_3=\rho_*$ and the fact that
   this induces an endomorphism of $A_1$ and $A_2$ to get commutativity. This greatly
   simplifies our proof and drastically reduces the resulting bounds.
   
   In Section~\ref{sec:tangent} we use primitivity of the CM type, via the tangent space,
   to show that primes with our type of reduction divide the exponent $n$ of the
   kernel of the isogeny
   $\overline{J}\rightarrow A_1^3$.
   This argument is exactly the same as in \cite{KLLNOS}, hence that section
   is very short and is basically a reference to \cite{KLLNOS}.
   In genus~$2$ \cite{GorenLauter}
   such an argument is not needed, see \cite[Section~5]{KLLNOS} for details.
   
   In Section~\ref{sec:endofproof} we bound the exponent $n$ mentioned
   in the previous paragraph.
   For this, we need to have a well-chosen isogeny in Section~\ref{sec:decomp}
   to base the exponent~$n$ on, and we need to look at what happens with the polarizations
   (which give rise to positive definite matrices) under our isogenies.
   This completes the proof of the first inequality of Theorem~\ref{thm:main}.
   
   Section~\ref{sec:geometryofnumbers} uses geometry of numbers
   to derive the second inequality from the first.

	\section{Reduction of Picard curves}\label{sec:reduction}

    In this section we give the explicit type of reduction that follows from a prime dividing the
    invariant $b$ of a Picard curve.
    Lemma~\ref{lem:cases} gives the three possible reduction types of the curve,
    and Proposition~\ref{prop:decomp} shows what this implies for the decomposition
    of the reduction of the Jacobian.
	
	\begin{lemma}\label{lem:cases}
	Let $C$ be a Picard curve of genus $3$
	over a number field $L$ and let $\mathfrak{p}\nmid 6$ be a prime of $\mathcal{O}_L$.
	Let $j = u/b^\exponentofinvariant$ be an absolute Picard curve invariant.
	
	If $\mathrm{ord}_{\mathfrak{p}}(j(C)) < 0$,
	then after replacing~$L$ with an extension and~$C$ with an isomorphic curve, 
	we are in one of the following cases:
	\begin{enumerate}
		\item $C : y^3 = x^4 + ax^2 + bx + 1$ with $a$, $b\in \mathcal{O}_L$
		such that
		$b\equiv 0$ and $a\equiv \pm 2$ modulo $\mathfrak{p}$.
		The reduction of this equation (from $\mathcal{O}_L$ to $\mathcal{O}_L/\mathfrak{p}$)
		is the singular curve $y^3 = (x^2\pm 1)^2$ of geometric
		genus~$1$;
		\item $C : y^3 = x^4 + x^2 + bx + c$ with $b$, $c\in\mathfrak{p}$.
		The reduction of this equation is the singular curve $y^3 = (x^2+1)x^2$ of geometric genus $2$;
		\item $C : y^3 = x^4 + ax^2 + bx + 1$ with $a$, $b\in\mathcal{O}_L$
		such that $b\equiv 0$ and $a\not\equiv \pm 2$ modulo $\mathfrak{p}$.
		The reduction of this equation is the smooth
		curve $y^3 = x^4 + \overline{a} x^2 + 1$ of genus $3$,
		where $\overline{a} = (a\ \mathrm{mod}\ \mathfrak{p})$.
	\end{enumerate}
	\end{lemma}

\newcommand{\case}[1]{case~\textit{(#1)}}
	\begin{proof}
	Let $m_0 = \mathrm{min}\{\frac{1}{2}v(a), \frac{1}{3} v(b), \frac{1}{4} v(c)\}$, where $v=\mathrm{ord}_{\mathfrak{p}}$ is the $\frakp$-adic valuation.

	By our assumption that $\mathfrak{p}$ divides the denominator of~$j(C)$,
	this minimum is not attained by $\frac{1}{3} v(b)$.
	
	If it is attained by $\frac{1}{4} v(c)$,
	then we scale the curve over $\overline{L}$
	as in \eqref{eq:isomorphic} so that $c=1$.
	As the minimum is not attained by $\frac{1}{3} v(b)$, we get
	that the reduction is $y^3 = x^4 + \overline{a} x^2 + 1$,
	where the right hand side has discriminant $16 (\overline{a}-2)^2(\overline{a}+2)^2$.
	In particular, if $\overline{a}\not=\pm 2$, then we are in \case{3}.
	
	If $\overline{a}= \pm 2$, then the reduction is
	$y^3 = (x^2\pm 1)^2$.
	Let $Y = (x^2\pm 1)/y$.
	Then we get $Y^3 = (x^2 \pm 1)$, that is,
	the curve $\overline{C}$ is birational to the elliptic curve $x^2 = Y^3 \mp 1$
	with $j$-invariant $0$.
	
The only remaining case is the case where the minimum is not attained by $\frac{1}{4} v(c)$.
As the minimum is not attained by $\frac{1}{3} v(b)$, we find that it is only
attained by $\frac{1}{2}v(a)$.
Now we scale the curve so that $a=1$.
We get that the reduction is $y^3 = x^4+x^2 = (x^2+1)x^2$.
Let $Y = y/x$.
Then we get $xY^3 = x^2+1$, which is
the hyperelliptic curve
$x^2 - Y^3 x = -1$ of genus two.
In fact, taking $X = 2x -  Y^3$, we get
the hyperelliptic curve $X^2 = Y^6-4$.
\end{proof}

	\begin{example} \label{example}
		Let $K= K_+(\zeta_3)$, where $K_+ = \Q[y]/ (y^3-y^2-4y-1) = \Q(\zeta_{13})_+$
		is totally real abelian of discriminant $13^2$ and conductor~$13$.
		Let $$ C:y^3 = x^4-2\cdot 7^2\cdot 13x^2+2^3\cdot 5\cdot 13 \cdot 47 x
		- 5^2\cdot 13^2 \cdot 31.$$
		The curve $C$ was computed by Koike and Weng~\cite[\S 6.1(3)]{KoikeWeng}, who
		conjecture that its Jacobian has CM by $\mathcal{O}_K$
		of primitive CM type.
		This curve and its reductions also appear in
		Bouw-Cooley-Lauter-Lorenzo-Manes-Newton-Ozman \cite[\S 5.2]{BCLLMNO}.
		
		We compute
		$$ j_1 = - \frac{7^6\cdot 13}{2^{3}\cdot 5^2\cdot 47^2},\quad
		   j_2 =   \frac{7^2\cdot 13\cdot 31}{2^5\cdot 47^2},\quad
		   j_3 = -\frac{5^2\cdot 13^2\cdot 31^3}{2^{12}\cdot 47^4}.
		   $$
		We find that the primes in the denominators of $j_1$, $j_2$, and $j_3$
        are $2$, $5$ and~$47$.
		Lemma~\ref{lem:cases} does not apply to the prime~$2$
		as it divides~$6$.
		The prime $5$ is of \case{2}.
		The prime $47$ is of \case{3} as follows:		
		Take an integer $r\equiv 11$ modulo $47$,
		let $\alpha = \sqrt{r}$ and $L = \Q(\alpha)$.
		Then $C$ is isomorphic over~$L$ to the curve given by
		$$ y^3 = x^4 -\alpha^6\cdot 2\cdot 7^2\cdot 13x^2+\alpha^9\cdot 2^3\cdot 5\cdot 13 \cdot 47 x
		- \alpha^{12}\cdot 5^2\cdot 13^2 \cdot 31,$$
		which modulo $47$ is
		$$ C : y^3 = x^4 + 19 x^2 + 1.$$
	\end{example}
	\begin{remark}\label{rem:reductiontypes}
		We know no examples of Picard curves with primitive CM
		by a sextic field that have a reduction as in \case{1}.
	\end{remark}
\begin{proposition}\label{prop:decomp}
	Let $C$ be a Picard curve of genus $3$ over a number field $L$
	containing a primitive third root of unity~$\zeta_3$.
	Let $\mathfrak{p}\nmid 6$ be a prime of~$L$.
	Suppose that $C$ is given (over $L$) by an equation as in one of the three
	cases in the conclusion of Lemma~\ref{lem:cases}.
	
	Let $J = \Jac(C)$ be the Jacobian of $C$, let $\Jcal$ be its
	N\'eron model over $\Z_{\mathfrak{p}}$ and let
	$\Jbar$ be its reduction modulo $\mathfrak{p}$.
	Assume that $J$ has CM or that we are in \case{3}. 
	Recall that $\zeta_3=\rho_*$ is a third root
    of unity in $\mathrm{End}(J)$; it induces
    endomorphisms of $\mathcal{J}$ and $\overline{J}$,
    which we also denote by~$\zeta_3$.
	
	Then there are abelian subvarieties $A_i$
	(with inclusion maps $I_i : A_i\hookrightarrow \Jbar$),
	surjective homomorphisms $s_i : \Jbar\rightarrow A_i$
	and
	endomorphisms $e_i\in\End(\Jbar)$ 
	for $i\in\{1,2\}$,	
	and an integer $d\in\{1,2\}$
	such that the following holds for all $i, j\in \{1,2\}$:
	\begin{itemize}
		\item[(a)]
		\begin{align*}
e_1+e_2 &= [d] & & \in \End(\Jbar),\\
e_i^2 &= [d] e_i & & \in\End(\Jbar),\\
e_1e_2 = e_2e_1 &= 0 & & \in\End(\Jbar),\\
e_i^{\!\!\textup{\dagger}} &= e_i & & \in \End(\Jbar),
\quad\mbox{where $\!{}^{\textup{\dagger}}$ denotes the Rosati involution,}\\
e_i &= I_i s_i & & \in\End(\Jbar),\\
s_i I_i &= [d] & & \in \End(A_i),\\
\mbox{if}\ i\not=j,\ \mbox{then}\quad s_i I_j &= 0 & & \in\Hom(A_j, A_i).
\end{align*}
Here and later, we write simply
$fg$ for $f\circ g$
in order to keep the notation clean and concise.
\item[(b)] The abelian variety $A_i$ has dimension
$i$ and we have a commutative diagram
\[
\xymatrix@R+1pc@C+1pc{\Jbar\ar[r]^{\left({s_1\atop s_2}\right)\quad\,}\ar@/_2pc/[rr]_{[d]} & A_1\times A_2 \ar[r]^{\quad(I_1\ I_2)}\ar@/_2pc/[rr]_{[d]}
	& \Jbar\ar[r]^{\left({s_1\atop s_2}\right)\quad}
	& A_{1}\times A_{2}.}
    \]
\item[(c)] if $i\not=j$, then we have $s_i \zeta_3 I_j = 0\in\Hom(A_j,A_i)$.
	\end{itemize}
\end{proposition}

We prove Proposition~\ref{prop:decomp} separately in the smooth
case and in the singular cases.
The smooth \case{3} is the main case, and the proof
is Section~\ref{sec:cover}. The singular
cases \textit{(1)} and~\textit{(2)} are proven in Section~\ref{sec:singular}.

	\subsection{The smooth case: $y^3=x^4+\overline{a}x^2+1$}
	\label{sec:cover}
    
    We now prove Proposition~\ref{prop:decomp} in 
    the smooth \case{3}, where we will see that
    it holds with $d = 2$.
	We consider the Picard curve $\Cbar:y^3=x^4+\overline{a}x^2+1$.
	The automorphism group $\Aut(\Cbar)$  contains the elements
	$\sigma:(x,y)\mapsto(-x,y)$ of order $2$ and 
	$\rho=\rho_{\Cbar} :(x,y)\mapsto(x,\zeta_3y)$ of order~$3$.

As $C$ has good reduction at~$\mathfrak{p}$,
we have $\Jbar = \Jac(\overline{C})$: we associate to $C$ the $\mathcal{O}_{K,\mathfrak{p}}$-scheme  $\operatorname{Pic}^0(\mathcal{C})$. By \cite[Theorem 9.3.7]{BLR}, the special fiber of $\operatorname{Pic}^0(\mathcal{C})$ is  $\Jac(\overline{C})$ and it is smooth. Finally, by \cite[Theorem 9.5.1]{BLR}, we have that $\operatorname{Pic}^0(\mathcal{C})$ is a Néron model $\mathcal{J}$ of the Jacobian of $\mathcal{C}$. In particular, the special fiber of $\operatorname{Pic}^0(\mathcal{C})$, i.e., $\Jac(\overline{C})$, is isomorphic to the special fiber of $\mathcal{J}$, that is, $\overline{J}$.

	The curve $\Cbar$ is a $2$-cover of the elliptic curve
	$E:\,v^2+av=u^3-1$ with CM by $\mathbb{Z}[\zeta_3]$.
	Indeed, we have
	\begin{align}\label{eq:phi}
	\phi:\qquad\Cbar&\longrightarrow E\\
	 (x,y)&\longmapsto(u,v)=(y,x^2).\nonumber
	\end{align}
	The curve $E$ also has an automorphism
	$\rho = \rho_E : (u,v)\mapsto (\zeta_3 u, v)$
	of order~$3$, and we have
	$\rho_E\circ\phi = \phi\circ \rho_{\overline{C}}$,
	or simply
	\begin{equation}\label{eq:rhophi}
	\rho\phi = \phi\rho.
	\end{equation}
	
	For every curve morphism $f:D_1\rightarrow D_2$,
	we get a pushforward morphism $f_*:\Jac(D_1)\rightarrow \Jac(D_2)$
	and a pullback morphism $f^* : \Jac(D_2)\rightarrow \Jac(D_1)$.
	With this notation, let
	$$e_1 = \phi^*\phi_*,\quad e_2 = [2]-e_1\quad \in\End(\Jbar).$$
	Let $A_i$ be the image of $e_i$ and let $s_i$ be defined
	by the commutative diagram
		    \[\xymatrix{\Jbar\ar@{->>}[r]_{s_i}\ar@/_2pc/[rr]_{e_i} & A_i \ar@{^{(}->}[r]_{I_i} & \Jbar}.\]
	Let $d = 2$.
	The equality $e_1+e_2 = 2$ is the definition of $e_2$.
	As $\phi$ is a $2$-cover, we get
	\begin{equation}
	\phi_*\phi^* = [2]\quad\in\End(E).
	\end{equation}
	In particular, we get
	$e_1^2 = \phi^*\phi_*\phi^*\phi_* = \phi^*[2]\phi_* = 2e_1$
	and $e_2^2 = 4-4e_1+e_1^2 = 2e_2$.
	We also get $e_1e_2 = e_1([2]-e_1) = 2e_1-2e_1=0$
	and similarly $e_2e_1=0$.
	
	By Mumford \cite[pages 327--328]{MumfordPrym},
	if $f : D_1\rightarrow D_2$ is a non-constant curve morphism and
    $(\Jac(D_i), \lambda_i)$ is the Jacobian of $D_i$
	with its polarization,
	then $$(f_*)^\vee = \lambda_1 f^* \lambda_2^{-1}
	\quad\mbox{and}\quad
	\lambda_i^\vee = \lambda_i.$$
	Taking duals, we also have
	$(f^{*})^{\vee} = \lambda_2^{\phantom{1}} f_* \lambda_1^{-1}$.
	
	In particular, we get
	\begin{align*}
	e_1^\dagger &=
	\lambda_{\overline{C}}^{-1} (\phi^*\phi_*)^\vee \lambda_{\overline{C}}^{\phantom{1}}
	= \lambda_{\overline{C}}^{-1} (\phi_*)^\vee (\phi^{*})^{\vee} \lambda_{\overline{C}}^{\phantom{1}}
	= \phi^*\phi_* = e_1
	\end{align*}
	and $e_2^\dagger = [2]_{\phantom{1}}^\dagger -e_1^\dagger = [2]-e_1=e_2$.
	
	The identities $e_i = I_is_i$ are the definition
	of~$s_i$.
	To compute $ s_i I_j$, we compose with the surjective
	map $s_j$ and the injective map $I_i$.
	If $i=j$, then we get
	$$I_i (s_i I_i) s_i = e_i^2 = 2e_i = I_i [2] s_i.$$
	By surjectivity of $s_j$ and injectivity of $I_i$, this gives
	$s_i I_i = [2]$.
	If $i\not=j$, then we get
	$$I_i (s_i I_j) s_j = e_i e_j = 0 = I_i [0] s_j,$$
	hence again by surjectivity and injectivity we get
	$s_i I_j = 0$.
	This proves~(a).
	
	Commutativity of the diagram follows from
	$I_1s_1+I_2s_2 = e_1+e_2=2$
	and
	the formulas for~$s_iI_j$. 
	The dimension of $A_1$ is the dimension of~$E$, which is~$1$.
	The commutativity of the diagram shows that $A_1\times A_2$
	has the same dimension as $J$, hence
	$A_2$ has dimension~$2$, which proves~(b).

Finally, we prove (c).
Since $I_i$ is injective and $s_j$ is surjective, it suffices to 
prove $I_i s_i \zeta_3 I_j s_j = 0$,
that is, $e_i \zeta_3 e_j = 0$.

Recall $\zeta_3 = \rho_*$, and by~\eqref{eq:rhophi},
we have $\phi_*\rho_* = \rho_*\phi_*$. Hence we get
\begin{align*}
e_1\zeta_3 e_1 &=  \phi^*\phi_*\rho_*\phi^*\phi_*\\
&=  \phi^*\rho_*\phi_*\phi^*\phi_*\\
&=  \phi^*\rho_*[2]\phi_*\\
&=  \phi^*\phi_*\rho_*[2] = 2e_1\zeta_3.
\end{align*}
In particular, we get $e_1\zeta_3 e_2 = 2e_1\zeta_3-2e_1\zeta_3 = 0$.

Also, we have $\rho^*\rho_* = 1$, so $\zeta_3 = (\rho^*)^{-1}$.
Therefore, we also have
\begin{align*}
e_1\zeta_3 e_1 &=  \phi^*\phi_*(\rho^*)^{-1}\phi^*\phi_*\\
&=  \phi^*\phi_*\phi^*(\rho^*)^{-1}\phi_*\\
&=  \phi^*[2](\rho^*)^{-1}\phi_*\\
&=  [2](\rho^*)^{-1} \phi^*\phi_* = 2\zeta_3 e_1.
\end{align*}
In particular, we get $e_2\zeta_3 e_1 = 2\zeta_3 e_1 - 2\zeta_3 e_1 = 0$.

This proves Proposition~\ref{prop:decomp} with $d=2$ in \case{3}.\qed

\begin{remark}
	We did not need to write $A_2$ as the Jacobian of an
	explicit
	curve for our work. However, for those who are interested,
	if $\Cbar : y^3 = x^4 + ax^2 + 1$ with $a\not = 0, -2, 2$ in a field
	of characteristic not $2$ or $3$, then a special case of
	Ritzenthaler-Romagny \cite[Theorem~1.1]{RitzenthalerRomagny}
	gives $\Jbar \sim E \times \Jac(H)$ with
	$E$ as in Section~\ref{sec:cover} and
	$$H : -a y^2 = (x^{2} + 2 x - 2) \cdot (x^{4} + 4 x^{3} + \left(2 a^{2} - 8\right) x - a^{2} + 4).$$
\end{remark}

\subsection{The singular cases}\label{sec:singular}

In the singular cases, by Kılıçer-Lauter-Lorenzo-Newton-Ozman-Streng~\cite[Theorem~1.1]{KLLNOS},
we already have a bound $p < \frac{1}{8} B^{10}$
under the hypotheses of Theorem~\ref{thm:main}.
However, we will see that we can do better.

In this section, we prove Proposition~\ref{prop:decomp}
in the singular cases \textit{(1)} and~\textit{(2)},
where we will see that it holds with $d=1$.
In case \textit{($g$)} for $g \in\{1,2\}$, let $A_g$ be the Jacobian
of the \emph{smooth} model $C_g$ of the curve of
geometric genus $g$ listed
in Lemma~\ref{lem:cases}\textit{($g$)}.

	Since the curve $C$ has CM, Proposition $4.2$ in
	Bouw-Cooley-Lauter-Lorenzo-Manes-Newton-Ozman~\cite{BCLLMNO}
	applies, so the
	reduction $\overline{\mathcal{C}}$ of a stable model $\mathcal{C}$ of $C$ is tree-like and the reduction~$\overline{J}$ of its Jacobian $J=\Jac(C)$
is the polarized product of the Jacobians of the irreducible components of $\overline{\mathcal{C}}$.

Then Corollary~4.3 of \cite{BCLLMNO} states that the reduction of the stable model is a union of either three smooth curves of genus 1 or a smooth curve of genus 1 and a smooth curve of genus~2.
By Lemma~\ref{lem:components_of_bad_red} of the appendix
(see also Corollary~\ref{cor:components_of_bad_red}),
one of these curves
is isomorphic to the curve~$C_g$.
	We conclude 
	that the reduction of the stable model
	is the union of a copy of $C_g$
	and up to two additional smooth curves of total genus $3-g$.
	Let $A_g$ be the Jacobian of $C_g$ and let
	$A_{3-g}$ be the polarized product of the Jacobians
	of those additional curves, so 
	\begin{equation}\label{eq:jbarisa1a2}
	\Jbar = A_1\times A_2
	\end{equation}
	as principally polarized abelian varieties.

	For $i\in\{1,2\}$,
	let $I_i$ be the inclusion map of $A_i$ into $\Jbar$
	and let $s_i$ be the projection map of $\Jbar$ onto~$A_i$.
	Let $e_i = I_is_i$. Then we get $s_iI_j = 0$ if $i\not=j$
	and $s_iI_i = [d]$ with $d=1$. As \eqref{eq:jbarisa1a2}
	is an identity
	of principally polarized abelian varieties, we 
	get $e_1^\dagger = e_1$,  $e_2^\dagger=e_2$, and $e_1+e_2=[1]$.
	The identities $e_i^2=e_i$ and $e_1e_2=e_2e_1=0$
	now follow from the identities in terms of $I_i$ and $s_i$,
	and the commutativity of the diagram follows
	from all the given identities.
	This proves (a) and (b). 
	
	Next we prove (c). As $I_i$ is an injective map and $s_j$ is a surjective one, it suffices to 
	prove $I_i s_i \zeta_3 I_j s_j = 0$,
	that is, $e_i \zeta_3 e_j = 0$.
	By the N\'eron mapping property, the automorphism $\rho$
	of $C$ uniquely extends to an automorphism of the stable model.
	And by the explicit equations in Lemma~\ref{lem:cases}, it also
	extends to an automorphism of order $3$ of $C_g$.
	Let $\zeta_3$ denote not only $\rho_*$ on $\Jbar$,
	but also $\rho_*$ on $A_c$.
Then we get $s_g\zeta_3=\zeta_3s_g$ and $\zeta_3I_g=I_g\zeta_3$. So, we get
	\begin{align*}
		e_g\zeta_3 e_g &= I_gs_g\zeta_3I_gs_g\\
		&=  I_gs_gI_gs_g\zeta_3\\
		&= I_gs_g\zeta_3=\zeta_3I_gs_g\\
		&=e_g\zeta_3=\zeta_3e_g.
	\end{align*}
	In particular, we get $e_{g}\zeta_3 e_{3-g} = e_g\zeta_3-e_g\zeta_3 = 0$ and $e_{g-3}\zeta_3 e_g = \zeta_3e_g-e_g\zeta_3e_g = 0$.
	
	This proves Proposition~\ref{prop:decomp} in cases \textit{(1)} and~\textit{(2)}.
	Case~\textit{(3)} was done in the previous section.\qed

\section{Decomposition and matrices}\label{sec:decomp}

If a prime $\mathfrak{p}$ does not divide~$6$ and does appear in the denominator
of~$j(C)$, then Section~\ref{sec:reduction} shows that $\overline{J}$
is isogenous (via the isogeny $F_0$) to a product of abelian varieties $A_1$ and $A_2$
of lower dimension. We also got lots of information about the isogeny~$F_0$,
and how it behaves with respect to the third root of unity
$\zeta_3=\rho_*\in \End(\overline{J})$ (see Proposition~\ref{prop:decomp}).

In this section we show that if~$J$ has complex multiplication,
then we can use an element~$\mu$ of the endomorphism ring of~$J$
to decompose $A_2$ further.

Just the fact that $A_2$ is decomposable is not enough.
In order to have small and explicit bounds in the end, it is crucial
that we can compute the degree of the isogeny $A_2\rightarrow A_1\times A_1$
in terms of elements of~$\mathcal{O}$.

So suppose from now on that we are in the situation
of the hypotheses of the main theorem (Theorem~\ref{thm:main}).
In other words, we have $\End(J_{\overline{L}})= \mathcal{O}$ for an order $\mathcal{O}$
in a sextic CM field~$K$, we have a totally real element $\mu\in \Z + 2\mathcal{O}\setminus\Z$,
an absolute Picard curve invariant~$j$, and 
a prime $\mathfrak{p}$ of $L$ lying over a rational prime~$p$
such that $\mathrm{ord}_{\mathfrak{p}}(j(C))<0$.

Suppose for now that $p\not=2,3$.

We get $\zeta_3\in \mathcal{O}$ (see Section~\ref{ssec:CM})
and hence $K = K_+(\zeta_3)$ for the totally real
cubic field $K_+$ of~$K$.

Our goal is only to bound $p$,
so without loss of generality we assume that all
elements of $\End(J_{\overline{L}})$
and the isomorphisms and models of
Lemma~\ref{lem:cases} are defined over~$L$.

	\begin{remark}
	In Kılıçer-Lauter-Lorenzo-Newton-Ozman-Streng~\cite{KLLNOS},
	a $\mu$ is taken with $\mu^2 \in K_+$ totally negative.
	In our situation, we can switch between totally negative and totally positive
	$\mu^2$ by replacing $\mu$ by $(2\zeta_3+1)\mu$, and the proof
	remains roughly the same. To
	make the proof as simple as possible,
	we will work with totally
	positive $\mu^2$, that is, totally real $\mu$.
\end{remark}

Let $F_0 = (I_1\ I_2) : A_1\times A_2\rightarrow \Jbar$
be the isogeny from Proposition~\ref{prop:decomp}(b),
and let $s_i$ and $d$ also be as in that proposition.
   We get an embedding
	\begin{align*}
	\iota_0: \End(\Jbar)\otimes \Q & \longrightarrow \End(A_1\times A_2)\otimes \Q,\\
	\alpha &\longmapsto F_0^{-1} \alpha F_0
	\quad =\quad
	\frac{1}{d} \left({s_1\atop s_2}\right)\circ  \alpha\circ (I_1\ I_2)
	= \frac{1}{d}
	\left(\begin{array}{cc}
	s_1\alpha I_1 & s_1\alpha I_2\\
	s_2\alpha I_1 & s_2\alpha I_2\end{array}\right)
	\end{align*}
	sending
	\begin{equation}\label{eq:firstembedding}
	\Z + 2\mathcal{O}\subset \Z + d\End(\Jbar)\rightarrow\End(A_1\times A_2).
	\end{equation}

	Write
	\begin{equation}\label{eq:iota0mu}
	\iota_0(\mu) = \xyzw,
	\end{equation}
	where the size of a box reflects the dimension
	of the domain and codomain of the homomorphism.
	As $\mu\in \Z+2\mathcal{O}$, by \eqref{eq:firstembedding},
	we get
	$x = \frac{1}{d}s_1\mu I_1\in \End(A_1)$,
	$y = \frac{1}{d}s_1\mu I_2 \in \Hom(A_2,A_1)$,
	$z=\frac{1}{d}s_2\mu I_1\in \Hom(A_1,A_2)$, and
	$w=\frac{1}{d}s_2\mu I_2\in\End(A_2)$.
	\begin{lemma}\label{lem:sqrtmin3}
	We have $$\iota_0(2\zeta_3+1) = \EndEA{r_1}{0}{0}{r_2},$$
	where $r_i\in\End(A_i)$
	satisfy $r_i^2=-3$.
	\end{lemma}
	\begin{proof}
		The off-diagonal boxes are zero by the equalities $s_i \zeta_3 I_j = s_i I_j=0$
		of Proposition~\ref{prop:decomp}(a,c).
		This gives the shape of the matrix. As its square is $-3$, we
		get $r_i^2 = -3$.
	\end{proof}
	
\begin{lemma}
	\label{lem:F1}
	The homomorphism 
\begin{align*}
F_1 =  \homEEEtoEA{1}{0}{0}{0}{z}{\!\! wz \!\!}
: A_1\times A_1\times A_1 & \rightarrow A_1\times A_2\\
(P,Q,R) &\mapsto (P, z(Q) + wz(R))
\end{align*}
is an isogeny.
\end{lemma}	
\begin{proof}
	It is necessary and sufficient to prove that
	the map $A_1\times A_1\rightarrow A_2$
	given by $(Q, R)\mapsto z(Q)+zw(R)$ is an isogeny.
	But this is analogous to \cite[Lemma~3.1]{KLLNOS},
	and the proof is identical.
	We only use that $\mu$ does not have degree $1$ or $2$ over~$\Q$.
\end{proof}

\begin{remark}
	An alternative choice of isogeny $F_1: A_1^3\rightarrow A_1\times A_2$
	is obtained by replacing $wz$ by
	$z' = \frac{1}{2} (zx+wz).$
	Indeed, write $\mu = i + 2j$ with $i\in\{0,1\}$ and $j\in\OO$.
	Then we get $\frac{1}{2}(\mu^2-i) = 2(j^2-ij)\in 2\OO$.
	As $z'$ is the lower left entry of $\iota_0(\frac{1}{2}(\mu^2-i))$,
	it is in $\Hom(A_1,A_2)$.
	Then instead of $F_1$ use
	\begin{align*}
	\homEEEtoEA{1}{0}{0}{0}{z}{z'}
	: A_1\times A_1\times A_1 & \rightarrow A_1\times A_2\\
	(P,Q,R) &\mapsto (P, z(Q) + z'(R)).
	\end{align*}
	This gives a bound in the end whose valuation at $2$ is better,
	but still non-optimal.
	As it makes the formulas more complicated, we
	we will not consider it further in this article,
	but we give this choice as an option in our SageMath implementation.
\end{remark}
Let $\mathcal{R} = \End(A_1)$ and $\mathcal{B}=\mathcal{R}\otimes\Q$.
We get an isogeny $F = F_0 F_1$ and ring homomorphisms
	\begin{align*}
	\iota_1\ :\  \End(A_1\times A_2) &\longrightarrow \Mat_{3\times 3}(\mathcal{B})\\
	 f&\longmapsto F_1^{-1} f F_1,\quad\quad\mbox{and}\\
	\iota = \iota_1\circ \iota_0\ :\ \End(\Jbar) &\longrightarrow \Mat_{3\times 3}(\mathcal{B}),\\
	 \alpha&\longmapsto F^{-1} \alpha F.
	\end{align*} 
	Take $n\in\Z_{>0}$ such that
	\begin{equation}\label{eq:conditiononn}
	[n]\ker(F_1)=0.
	\end{equation}
 In \eqref{eq:explicitn} below,
	we will take a specific $n$.
	
\section{Using commutativity to get matrices over a field}	\label{sec:commute}
	
	In this section we use the fact that we have an explicit
	$\zeta_3$ that commutes with $\mu$ in order to find
	that the entries of the $3\times 3$ matrix $\iota(\mu)$
	from Section~\ref{sec:decomp} all lie in the same
	quadratic field.
	In the proof of the previous bounds
	(Goren-Lauter~\cite{GorenLauter} for $g=2$ and
	\cite{BCLLMNO, KLLNOS}
	for $g=3$),
	we had no such $\zeta_3$, and
	the proof that the entries were in a quadratic field
	was based instead on the fact
	that ``small'' elements of large-discriminant quaternion algebras commute,
	hence that argument worked only for very large primes.
	Because of our explicit decomposition,
    our proof is much simpler and our results are much sharper.

	We also get various relations between the entries, which we
	use in Section~\ref{sec:endofproof} to bound the entries.
	
	\begin{lemma}\label{lem:matrix}
		For $\mu$ and $\iota$ as in Section~\ref{sec:decomp}
		and every $\alpha\in \Z+2\mathcal{O}$, 
		the following hold:
		\begin{enumerate}
			\item the entries of the $3\times 3$ matrix $\iota(\alpha)$ are in $\frac{1}{n}\mathcal{R}$, and the entries of the top row are
			in~$\mathcal{R}$, where $\mathcal{R}=\End(A_1)$,
\item we have
		$$
		\iota(\mu)=\left(\begin{array}{ccc}x & a & b \\ 1 & 0 & e \\ 0 & 1 & f\end{array}\right),
		$$
		with $x$, $a$, $b$, $ne$, $nf\in \mathcal{R}$, and
\item we have
		$$
\iota(2\zeta_3+1)=\left(\begin{array}{ccc}r_1 & 0 & 0 \\ 0 & s & t \\ 0 & u & v\end{array}\right)
$$
		with $r_1$, $ns$, $nt$, $nu$, $nv\in\mathcal{R}$ and $r_1^3 = -3$.
		\end{enumerate}
	\end{lemma}
	\begin{proof}
Let $G:A_1\times A_2\rightarrow A_1\times A_1\times A_1$ be the isogeny satisfying $GF_1=[n]$,
which exists because of~\eqref{eq:conditiononn}.
Then we have
$$F_1 = \homEEEtoEA{1}{0}{0}{0}{z}{wz}\qquad
\mbox{and}\qquad
F_1^{-1} = \frac{1}{n}G = \frac{1}{n}\homEAtoEEE{n}{0}{0}{g_1}{0}{g_2}$$
for some~$g_i:A_2\rightarrow A_1$
satisfying $\frac{1}{n}({g_1\atop g_2})(z,wz) = ({1\atop 0}{0\atop 1})$.

		For $i,j\in\{1,2\}$, the $(i,j)$-entry of $\iota_0(\alpha)$
		is in $\Hom(A_j, A_i)$.
		Now, because of the shape of $F_1^{\vphantom{-1}}$ and $F_1^{-1}$,
		the matrix $\iota(\alpha) = F_{1}^{-1} \iota_0(\alpha)F_{1}^{\vphantom{-1}}$
		has entries in $\frac{1}{n}\mathcal{R}$
		with the entries of the top row in $\mathcal{R}$. This proves~\textit{(1)}.
		
For \textit{(2)}, we now only have to compute the lower left $2\times 2$ block,
so	
$$
\iota(\mu)=F_1^{-1}\iota_0(\mu)F_1^{\vphantom{-1}}=\frac{1}{n}\homEAtoEEE{*}{*}{0}{g_1}{0}{g_2}\EndEA{*}{*}{z}{w}\homEEEtoEA{1}{0}{*}{0}{z}{*}=\left(\begin{array}{ccc}* & * & *\\ 1 & 0 & *\\ 0 & 1 & *\end{array}\right).
$$
For \textit{(3)}, we note that by Lemma~\ref{lem:sqrtmin3}
we are multiplying block-diagonal matrices as follows:
\[
\iota(2\zeta_3+1)=F_1^{-1}\iota_0(2\zeta_3+1)F_1^{\vphantom{-1}}=\homEAtoEEE{1}{0}{0}{*}{0}{*}\EndEA{r_1}{0}{0}{*}\homEEEtoEA{1}{0}{0}{0}{*}{*}
=\left(\begin{array}{ccc}r_1 & 0 & 0\\ 0 & * & *\\ 0 & * & *\end{array}\right).
\qedhere
\]
	\end{proof}
\begin{remark}
	Lemma~\ref{lem:matrix}\textit{(2)} and its proof are
	analogous to Lemma 3.2 of
	Kılıçer-Lauter-Lorenzo-Newton-Ozman-Streng~\cite{KLLNOS}
	and its proof.
\end{remark}

The following lemma is one of the things that distinguishes our proof
from the proofs of earlier denominator bounds. It shows that
all entries of the matrices in Lemma~\ref{lem:matrix}
commute.
Contrary to the previous bounds,
it shows this without the need
for using that small elements in large-discriminant
quaternion rings commute, and hence without assuming
that $p$ is large.

\begin{lemma} \label{lem:Q(r)}
In the notation of Lemma~\ref{lem:matrix},
we have $v=s=r_1$ and $u=t=0$. Moreover, all of $x$, $a$, $b$, $e$, $f$
and all entries of $\iota(\alpha)$ for all $\alpha\in K$
are in $\Q(r_1)$.
\end{lemma}

\begin{proof}
As the matrices $\iota(\mu)$ and $\iota(2\zeta_3+1)$
commute, we have 
$$
\left(\begin{array}{ccc}r_1x & r_1a & r_1b \\ s & t & se+tf \\ u & v & ue+vf\end{array}\right) = \left(\begin{array}{ccc}xr_1 & as+bu & at+bv \\ r_1 & eu & ev \\ 0 & s + fu & t+fv\end{array}\right).
		$$
		We immediately read off $s=r_1$ and
		$u=0$.
		And once we use $u=0$,
		we also get
		$t=0$ and $v=s$.
		Now $\iota(2\zeta_3+1)$ is $r_1$ times the identity matrix,
		hence the fact that the two matrices commute implies that
		all entries of the matrices commute with $r_1$.
		
		As $r_1$ is not in $\Q$, this implies that these
		entries are in the quadratic field $\Q(r_1)$.
		
		Finally, as $\mu$ and $2\zeta_3+1$ generate
		the field $K$, we get that all entries
		of $\iota(\alpha)$ are in~$\Q(r_1)$.
\end{proof}	

In the rest of this section, we express $b$, $e$, and $f$ in terms of
$x$, $a$, and the coefficients of the minimal polynomial of~$\mu$.

	As $\mu$ is cubic integral over $\Z$, we have
\begin{equation}
\mu^3 - t_1 \mu^2 + a_1 \mu - N = 0,
\end{equation}
where $t_1 = \tr_{K_+/\Q}(\mu)$, $N = N_{K_+/\Q}(\mu)$, and $a_1$
are in $\Z$ and depend only on~$\mu$.
\begin{lemma}\label{lem:expressbCD}
	We have
	\begin{equation}
	\begin{aligned}
	f &= t_1 - x,\\
	e &= - (a_1 + x^2 + a - t_1x),\\
	b &= N - (x^3 - t_1x^2 + 2xa + a_1x - t_1a).
	\end{aligned}
	\end{equation}
\end{lemma}
\begin{proof}
	As $\iota$ is a ring homomorphism, we
	find that the matrix $M = \iota(\mu)^3 - t_1\iota(\mu)^2 +a_1\iota(\mu) - N \identitymatrix{3}$
	is the zero matrix, where $\identitymatrix{3}$ is the $3\times 3$ identity matrix.
	
	As the entries of $\iota(\mu)$ are given explicitly
	in terms of $x, a, b, e, f$ in a \emph{field} $\Q(r_1)$,
	we can easily compute $M$ in terms of these quantities and $t_1, a_1, N$.
	The leftmost column is exactly
	\begin{equation}
	\left(\begin{array}{r}
	x^{3} -  t_1 x^{2} + \left(2 a + a_{1}\right) x - t_1 a + b -  N \\
	x^{2} -  t_1 x + a + e + a_{1} \\
	x + f -  t_1
	\end{array}\right),
	\end{equation}
	which proves the result.
\end{proof}

\section{Tangent spaces and primitive CM types}\label{sec:tangent}

As in \eqref{eq:conditiononn}, let $n\in\Z_{>0}$ be such that $[n]\ker(F_1)=0$.
In this section, we prove the following proposition, which implies
that in order to bound~$p$, it suffices to find a small~$n$.
In \eqref{eq:explicitn} below, we choose a specific~$n$.

\begin{proposition}\label{prop:pdividesn}
	For $C$ and $p$ as in the hypotheses of
	Theorem~\ref{thm:main}, let $n\in\Z_{>0}$ be 
	such that $[n]\ker(F_1)=0$.
	Then $p\leq 3$ or $p\mid n$.
\end{proposition}
\begin{proof}
	Suppose $p\nmid 6n$.
	We claim that primitivity of the CM type
	implies that the matrix $\iota(2\zeta_3+1)$
	has two distinct eigenvalues.
	
	Note that having two distinct eigenvalues contradicts the first statement
	of Lemma~\ref{lem:Q(r)}, which was the equality
	$\iota(2\zeta_3+1) = r_1 \identitymatrix{3}$.
   In particular, the result follows once we prove the claim.
	
	The idea behind the claim is as follows.
	Note that 
	primitivity of the CM type implies that the action
	of $2\zeta_3+1$ on the tangent space of $J$ has
	two distinct eigenvalues.
	If $p$ does not divide $6n$, then these two eigenvalues
	induce distinct eigenvalues for the action on the tangent
	space of $\Jbar$ via $F$ and $[2n]F^{-1}$.
	This proves the claim.
	
	In more detail,
	the proof of the claim is the
	same as the proof of Proposition 5.8 of
	Kılıçer-Lauter-Lorenzo-Newton-Ozman-Streng~\cite{KLLNOS}
	with $\delta=3$ and $\sqrt{-\delta} = 2\zeta_3+1$,
	except for the following changes:
	\begin{enumerate}[label=(\alph*)]
		\item We use $F$ as above and we use $A_1$ instead of $E$ and
		$2n$ instead of $n$.
		Let $G = [2n]F^{-1}$.
		\item
	Instead of \cite[Proposition 4.1]{KLLNOS}, use
	Lemma~\ref{lem:Q(r)}, so in particular the condition
	$p>\frac{1}{8}B^{10}$ is not needed.
	\item The reductions of $\pm \sqrt{-\delta}$
	are distinct as we have $p\nmid 2\delta=6$. This
	also does not need any additional bounds on $p$.
	\item The invertibility of $2n$ modulo $p$ follows from
	our assumption $p\nmid 6n$ and also does not need
	additional bounds on~$p$.\qedhere
	\end{enumerate}
\end{proof}

	\section{Using the polarization to get bounds}
	\label{sec:endofproof}
	
By Proposition~\ref{prop:pdividesn}, it now suffices to find a
sufficiently well-bounded $n\in\Z_{>0}$ with $[n]\ker(F_1)=0$.
In this section, we do exactly this, using the polarization that $C$ induces on $A_1^3$ via~$F_1$.
The key here is that we constructed $F_1$ very explicitly, and that polarizations give rise to positive definite
matrices. Compared to \cite{KLLNOS}, our matrices are a bit simpler, since
in our situation we are able to prove that the entries are in a field, where \cite{KLLNOS}
needs the bounds in order to prove exactly that. 

Let $\lambda = F^{\vee}\lambda_C F$ 
be the polarization
induced on $A_1^3$ by the polarization $\lambda_C$ of $\Jbar$.
We identify~$A_1$ with its dual via the natural polarization $\lambda_{A_1}$,
which we sometimes leave out from the notation.
Then $\lambda$ can be viewed as an endomorphism of $A_1^3$,
and the following result gives it as a matrix.
\begin{lemma}[{cf.~\cite[{Lemma~4.3}]{KLLNOS}}]
	\label{lem:lambda}
	We have
	\[ \lambda = \left(\begin{array}{ccc}
	m & 0 & 0 \\
	0 & \alpha & \beta\\
	0 & \beta^\vee & \gamma\end{array}\right).\]
	with $m, \alpha, \gamma \in\Z_{>0}$ and
	$\beta\in\mathcal{R}$ with
	$\alpha\gamma - \beta\beta^\vee > 0$.
	Moreover, we have $m\mid 2$.
\end{lemma}
\begin{proof}
	Recall from just above the statement of the lemma that $\lambda$ is defined as a homomorphism $A_1^3\rightarrow (A_1^\vee)^3$
	by $\lambda =F^\vee\lambda_C F$, and as an endomorphism of $A_1^3$
	by $\lambda = (\lambda_{A_1}^{-1}\times \lambda_{A_1}^{-1}\times \lambda_{A_1}^{-1}) F^\vee\lambda_C F$.
	The symmetry of $\lambda$
	now follows from the symmetry of $\lambda_C$, which is Mumford
	\cite[(3) on page 190]{Mumford}.
	
	We now prove that the off-diagonal entries of the first row
	and column of $\lambda$ are zero. Since~$F=F_0 F_1$, we write
\begin{equation}
\lambda = 
\mathrm{diag}(\lambda_{A_1}^{-1}
,\lambda_{A_1}^{-1},\lambda_{A_1}^{-1})
F_1^{\vee} (I_1\ I_2)^{\vee} \lambda_C (I_1\ I_2) F_1,
\qquad\mbox{where}\qquad
F_1^{\vee}=\homEAtoEEE{1}{0}{0}{z^\vee}{0}{\!\! z^\vee w^\vee \!\!}.
\end{equation}
To see that four entries are zero, we only look at the off-diagonal entries
of the first row. This suffices by symmetry.
By Proposition~\ref{prop:decomp}(a), we get
$e_1^\vee\lambda_C e_2^{\vphantom{\vee}}
= \lambda_Ce_1^\dagger e_2^{\vphantom{\dagger}}
= \lambda_C e_1^{\vphantom{\vee}} e_2^{\vphantom{\vee}} = 0$.
As~$e_i = I_is_i$ and
$s_i$ is surjective
we get $I_1^\vee \lambda_C I_2^{\vphantom{\vee}}=0$.
Therefore we have \[(I_1\ I_2)^{\vee} \lambda_C (I_1\ I_2) = \EndEA{*}{0}{*}{*},\]
and hence the off-diagonal entries of the first row of $\lambda$ are zero.

From the final paragraph of Application III on page 210 of Mumford~\cite{Mumford},
we get that $\lambda$ is positive definite, hence $m, \alpha, \gamma, \alpha\gamma-\beta\beta^\vee > 0$.

It remains only to prove $m\mid 2$.
We have $m = I_1^\vee \lambda_C I_1^{\vphantom{\vee}}$ since we defined $m$
to be the first diagonal entry of $(I_1\ I_2)^{\vee} \lambda_C (I_1\ I_2)$.

Recall that by Proposition \ref{prop:decomp}(a) we have
$e_1^{\vphantom{\dag}}=e_1^{\dag}$.
This implies $e_1=\lambda_C^{-1}e_1^{\vee}\lambda_C^{\vphantom{-1}}$
and by $e_1=I_1s_1$, we get $I_1s_1=\lambda_C^{-1}s_1^{\vee}I_1^{\vee}\lambda_C$.
Therefore, we have
$\lambda_CI_1s_1I_1=s_1^{\vee}I_1^{\vee}\lambda_CI_1^{\vphantom{\vee}}$,
hence $\lambda_C I_1[d]=s_1^{\vee}I_1^{\vee}\lambda_CI_1^{\vphantom{\vee}}$.
Since $\lambda_C$ is an isomorphism and $I_1$ is injective, we get that
$\ker(s_1^{\vee}I_1^{\vee}\lambda_CI_1)= A_1[d]$.
Hence, $\ker(I_1^{\vee}\lambda_CI_1)\subseteq A_1[d]$,
and we know that $m=I_1^{\vee}\lambda_CI_1$
is a positive integer. So we finally get $m=1$ or~$2$.
\end{proof}

Since $\mu\in K_+$, it equals its
complex conjugate~$\overline{\mu}$.
Moreover
(analogously to Proposition 4.8 of~\cite{BCLLMNO}),
we have for every $\eta\in K$,
\begin{equation*}
\lambda^{-1}\iota(\eta)^\vee\lambda
= (F^\vee\lambda_C F)^{-1}(F^{-1} \eta F)^\vee F^\vee \lambda_C F
= F^{-1}\lambda_C^{-1}\eta^\vee \lambda_C F = \iota(\eta^\dagger) = \iota(\overline{\eta}),
\end{equation*}
hence $\iota(\mu)^\vee \lambda = \lambda\iota(\mu)$, so
\begin{equation}\label{eq:conjugate}
\left(\begin{array}{ccc} mx^{\vee} & \alpha & \beta \\
m a^{\vee} &\beta^{\vee}&\gamma\\
mb^{\vee} & e^{\vee}\alpha+f^{\vee}\beta^{\vee}&
e^{\vee}\beta+f^{\vee}\gamma\end{array}\right)
=
\left(\begin{array}{ccc}m x  & ma & mb\\
\alpha & \beta & \alpha e+\beta f\\
\beta^{\vee}& \gamma & \beta^{\vee}e+\gamma f\end{array}\right).
\end{equation}
This tells us
\begin{equation}\label{alpha eq ma}
\begin{aligned}
\alpha &= ma, & & \mbox{hence $a\in \Q_{>0}\cap \mathcal{R} = \Z_{>0}$}; \\ 
\beta^\vee = \beta &= mb;\\
x &= x^\vee, & & \mbox{hence $x\in\Z$};\\
\gamma &= \alpha e + \beta f.
\end{aligned}
\end{equation}
Combining this with Lemma~\ref{lem:expressbCD},
we find explicit expressions for all entries of $\iota(\mu)$ and $\lambda$
in terms of $x$, $a$, $m$, and the coefficients of the minimal
polynomial of~$\mu$.
In particular, these entries are all in~$\Z$.

Let
\begin{equation}\label{eq:explicitn}
n = \alpha\gamma-\beta\beta^\vee = m(a\gamma-mb^2)\in m\Z.
\end{equation}
Then Lemma~\ref{lem:lambda} and the definition of $\lambda$
give
\[\left(\begin{array}{ccc} n/m& 0&0\\0& \gamma & -\beta \\
0 & -\beta^\vee & \alpha\end{array}\right)F^{\vee}\lambda_C
F = [n],\]
so that in particular the condition
$[n]\ker(F)=0$ from \eqref{eq:conditiononn} is satisfied.
We have already expressed $\alpha$, $\gamma$, and $\beta$
in terms of $x$, $a$, $m$, and the coefficients of the
minimal polynomial of~$\mu$. As $m$ is $1$ or~$2$, it suffices to  
bound $x$ and $a$
in order to bound~$n$.

As the $3\times 3$ matrix $\iota(\mu^2)$ over $\Q$
satisfies the (cubic) minimal polynomial of $\mu^2$ over $\Q$,
we find that its (matrix) trace is the trace of $\mu^2$ from $K_+$ to $\Q$,
which is
$t_2 := t_1^2-2a_1$.
We get
\begin{align}
t_2 &= x^2 + 2a + 2e + f^2 \nonumber\\
         &= x^2 + 2a + \frac{2}{\alpha} \gamma - \frac{2}{\alpha} \beta f + f^2 \nonumber\\ 
         &= x^2 + 2a + \frac{2}{\alpha} \gamma - (\frac{\beta}{\alpha})^2 + (\frac{\beta}{\alpha}-f)^2 \nonumber \\ \label{ineq with n}
        &= x^2 + 2a + \frac{\gamma}{\alpha} + \frac{n}{\alpha^2} + (\frac{\beta}{\alpha}-f)^2 \\ 
         &\geq x^2 + 2a. \nonumber
\end{align}

In particular, we get 
\begin{equation}\label{eq:boundax}
\begin{aligned}
|x| &\leq \sqrt{t_2}\quad\mbox{and}\\ 
0 < a &\leq \frac{1}{2}(t_2-x^2).
\end{aligned}
\end{equation}
Moreover, by \eqref{ineq with n}, we get $n \leq t_2 \alpha^2$ and $2a \leq t_2$.
Then by (\ref{alpha eq ma}), we obtain $n \leq t_2\alpha^2 \leq t_2 m^2a^2\leq t_2^3$ as~$m|2$.
By Proposition~\ref{prop:pdividesn}, we have $p\leq 3$ or
$p\mid n$.
Hence we get the bound 
$p\leq \text{max}\{3,\,t_2^3\}$.

\begin{lemma}\label{lem:traceatleast2}
	Let $\mu$ be a totally real cubic algebraic integer, and let $t_2$ be the trace
	of $\mu^2$. Then we have $t_2\geq 2$.
\end{lemma}
\begin{proof}
	Let $a$, $b$, $c$ be the images of $\mu$ under the three embeddings into $\R$.
	Then $t_2 = a^2 + b^2 + c^2\in\Z$.
	Suppose $t_2 < 2$. Then $t_2 \leq 1$ and $a^2$, $b^2$, $c^2>0$, hence $|a|$, $|b|$, $|c|\in (0,1)$, so we get $|abc|\in(0,1)$.
	On the other hand, we have $|abc| = |N(\mu)|\in \Z$. Contradiction.
\end{proof}
\begin{proof}[Proof of the first inequality in Theorem~\ref{thm:main}]
As stated above Lemma~\ref{lem:traceatleast2}, we have proven
the inequality $p\leq \text{max}\{3,\,t_2^3\}$
under the hypotheses of Theorem~\ref{thm:main}.
As Lemma~\ref{lem:traceatleast2} gives $t_2^3\geq 2^3>3,$
we get $p\leq t_2^3$.
\end{proof}

\section{Intrinsic bounds from geometry of numbers}\label{sec:geometryofnumbers}

At the end of Section~\ref{sec:endofproof}, we finished
the proof of the first inequality in Theorem~\ref{thm:main}:
$p\leq \tr_{K^+/\Q}(\mu^2)^3$.
Next, we show that there exists an element $\mu$
for which this right hand side is explicitly
bounded in terms of the discriminant
of~$K^+$,
and hence we prove the rest of Theorem~\ref{thm:main}.

Let $\{\sigma_1$, $\sigma_2$, $\sigma_3\}$ be the set of the real embeddings of $K_+$.
This gives us the map $\sigma : K_+ \rightarrow \R^3$ by sending $y$ to $(\sigma_i(y))_i$.
The order
$\Z+2\O_+\subset K_+$ is a lattice of co-volume $2^{3-1}|\Delta(\O_+)|^{1/2}$
in $\R^3$.
Let $R=4\pi^{-1/2}|\Delta(\O_+)|^{1/4}+\epsilon$ for some $\epsilon>0$.

We choose a symmetric convex body in $\R^3$:
$$
\calC_R = \{x\in \R^3: |x_1| < 1, x_2^2+x_3^2<R^2\}.
$$

We then have $\text{vol}(\calC_R) = 2\pi R^2>32|\Delta(\O_+)|^{1/2}
= 2^3\text{covol}(\Z+2\O_+)$.
 By Minkowski's first convex body theorem (see Siegel~\cite[Theorem 10]{Siegel}),
there is a non-zero $\mu\in(\Z+2\O_+)\cap \calC_R$.
Note that $\mu$ generates $K_+$:
if $\mu\in \Q$ then $\mu\in\Z$, but $|\mu|< 1$,
so we get $\mu=0$ which is a contradiction.

Then we get 
$$
\tr_{K_+/\Q}(\mu^2)= \sum_i\sigma_i(\mu^2) \leq (1+R^2).
$$
Since $\mu$ is an algebraic integer in~$K_+$, we have $\tr_{K_+/\Q}(\mu^2)\in\Z$. So when
we let $\epsilon$ tend to $0$, we get $t_2 = \tr_{K_+/\Q}(\mu^2)\leq (1+\frac{16}{\pi}|\Delta(\O_+)|^{1/2})$.

Since $p\leq t_2^3$ and $|\Delta(\O_+)|\geq 2$, we get
$p\leq (1+\frac{16}{\pi}|\Delta(\O_+)|^{1/2})^{3}< 196|\Delta(\O_+)|^{3/2}$.
This finishes the proof of Theorem~\ref{thm:main}.\qed

\section{Computing the set of primes}\label{sec:setofprimes}\label{sec:conjecture}

From the proof of Theorem~\ref{thm:main},
we get much more than just a bound on~$p$
as follows.
Take a totally real
$\mu\in \Z+2\mathcal{O}$. Then list all $a$ and $x$
satisfying the bounds of \eqref{eq:boundax}
and all $m\in\{1,2\}$.
For each, compute $n = n(\mu, a, x)$
using \eqref{alpha eq ma} and~\eqref{eq:explicitn}.
Then let $N_{\mu}$ be the product of the numbers $n(\mu, a, x)$.
Then $p$ divides $6N_\mu$ by Proposition~\ref{prop:pdividesn}.
This is already much better than just a bound on~$p$.

However, we can do even better.
For each $\mu, a, x, m$, we get a $\Q$-algebra homomorphism
$\iota : K\rightarrow\Mat_{3\times 3}(\Q(\zeta_3))$
(given on generators in Lemma~\ref{lem:matrix},
coefficients in $\Q(\zeta_3)$ by Lemma~\ref{lem:Q(r)}
and $r_1^2=-3$).
We know by Lemmas \ref{lem:matrix} and~\ref{lem:Q(r)}
that all elements of the image of $\Z+2\mathcal{O}$ are
matrices with entries in $\frac{1}{n}\Z[\zeta_3]$, with the
entries of the top row in $\Z[\zeta_3]$.
So we compute a $\Z$-basis of $\Z+2\mathcal{O}$ and throw away
all triples $(x,a,m)$ for which an element of this basis maps to a matrix
that does not satisfy the integrality condition.
We also throw away all triples $(x,a,n)$ for which one of $\alpha,\beta,\gamma$
is non-integral or for which $\gamma$ or $n$ is non-positive.
By making the set of pairs $(x,a)$ smaller in this way,
the product $N_{\mu}$ of the numbers $n(\mu, a, x)$ becomes much smaller.

We implemented the computation of $N_{\mu}$ in SageMath~\cite{sage}
and made the implementation available at~\cite{ourcode}.

\begin{theorem}\label{thm:main2}
	Let $C$ be a Picard curve of genus $3$ 
	over a number field $L$ and suppose that the
	endomorphism ring $\End(J_{\Lbar})$ of $J = \Jac(C)$
	over the algebraic closure is isomorphic to
	an order $\O$ of a number field $K$ of degree~$6$.
	
	Let $K_+$
	be the real cubic subfield of $K$ and $\O_+ = K_+ \cap \O$. 
	Let $\mu$ be a totally real element in~$\Z+2\O_+$ such that 
	$K = \Q(\mu)(\zeta_3)$.
	
	Let $j = u/b^\exponentofinvariant$ be an absolute Picard curve invariant.
	Let $\mathfrak{p}$ be a prime of $\O_L$ lying over a rational prime~$p$.
	If $\mathrm{ord}_{\mathfrak{p}}(j(C)) < 0$,
	then $p$ divides the number $6N_{\mu}$ for $N_{\mu}$ as
	described in the preceding paragraphs.\qed
\end{theorem}

\begin{conjecture}\label{conj:valuations}
	There are constants $s, e\in\Q_{>0}$ such that the following holds.
	Let $j=u/b^\exponentofinvariant$ be an absolute Picard curve invariant.
		
	Let $\mathcal{O}$ be an order in a sextic CM field. Let $\mathrm{CM}_K$ be the set of
	isomorphism classes of Picard
	curves $C$ over $\Qbar$ of genus $3$ with $\End(J_{\Qbar})$
	isomorphic to~$\mathcal{O}$.
	Let $N_{\mu}$ be as in Theorem~\ref{thm:main2}.
	
	Then for all non-archimedean valuations $v$ of $\Qbar$, we have
	\begin{equation}\label{eq:valuations}
	\sum_{C\in \mathrm{CM}_K} \max\{0,v(j(C))\} \leq \exponentofinvariant(v(s) + ev(N_{\mu})).
	\end{equation}
	\end{conjecture}
\begin{remark}\label{rem:cube}
	In fact, the examples in Section~\ref{sec:examples} below suggest
	that when $K/\Q$ is Galois, the constant $e=1/3$ suffices.
	
	The numerology that supports the factor $1/3$ is that
	$K$ has three CM types up to complex conjugation
	that are all equivalent, so that
	every curve should be counted three times,
	but is only counted once in the left hand side of~\eqref{eq:valuations}.
\end{remark}
To prove the conjecture, one would need to retrace our proof, but
working over prime-power quotients of $\mathcal{O}_L$ instead
of over the field $\mathcal{O}_L/\mathfrak{p}$.
Once the conjecture is proven, our implementation of $N_{\mu}$, together
with an interval-arithmetic-version of Lario-Somoza~\cite{LarioSomoza}
would give a proven algorithm for computing CM Picard curves
and Picard class polynomials. In particular, it would prove
the conjectured CM curves of Koike-Weng~\cite{KoikeWeng} and Lario-Somoza~\cite{LarioSomoza}.

\section{Examples}\label{sec:examples}

Finally, we take a few example curves and compare our bounds with previous bounds, and compare our invariants
with previous choices.

Given a Picard curve $C$, let $\den_1$ and $\den_3$ be the denominators
of the absolute invariants $j_1(C)=(a^3/b^2)(C)$ and $j_3(C)= (c^3/b^4)(C)$, respectively.
Then we define the \textit{absolute denominator} $\denabs$ of $C$ by
$$\denabs = \prod_{p|\den_1\cdot \den_3} p^{\max\{{v_p(\den_1)/2}, {v_p(\den_3)/4}\}}
= \prod_{p} p^{{v_p(b) - \frac{1}{4}\min\{6 v_p(a), 4 v_p(b), 3v_p(c)\}}}
.$$ 

Theorem~\ref{thm:main2} tells us that all primes dividing $\denabs^4$ in fact divide $6N_{\mu}$.
In fact, Conjecture~\ref{conj:valuations} implies that
$\denabs$ divides $sN_{\mu}^e$.

We define the \emph{absolute denominator $\denKWabs{}$ of the
Koike-Weng invariants} $j'_1 = b^2/a^3$ and $j'_2 = c/a^2$ in the same way.
In other words, let $\den'_1$ and $\den'_2$ be the denominators of
$j'_1(C)$ and $j'_2(C)$ respectively. Then we define
 
 $$\denKWabs{} = \prod_{p|\den'_1\cdot \den'_2}
  p^{\max\{{v_p(\den'_1)/3},{v_p(\den'_2)/2}\}}.$$ 

Let $\Delta$ be the discriminant invariant~\eqref{eq: disc. inv.}
on page~\pageref{eq: disc. inv.}, which has
weight $12$. We define the invariants
$$
i_1= \frac{a^6}{\Delta},\quad
i_2=\frac{a^3b^2}{\Delta},\quad
i_3=\frac{a^4c^{\vphantom{1}}}{\Delta},\quad
i_4= \frac{b^4}{\Delta},\quad
i_5= \frac{c^3}{\Delta}
$$
denoted $j_*$ in 
Kılıçer-Lauter-Lorenzo-Newton-Ozman-Streng
\cite{KLLNOS}.
Let $\denDeltaabs$ denote the least common multiple
of the denominators of $i_1(C)$, $i_4(C)$ and $i_5(C)$
(equivalently, of all $i_*(C)$).

Let \begin{equation}\label{eq:defB}
B = \min\{\tr_{K_+/\Q}(\alpha\overline{\alpha}): \alpha\in \mathcal{O}_K\setminus\{0\}, \overline{\alpha} = -\alpha \},
\end{equation}
so \cite[Remark~1.6]{KLLNOS} conjectures that primes $p$ of bad reduction
are $< \frac{1}{8} B^{10}$,
and in the case where $K/\Q$ is cyclic and
$C$ has CM by $\mathcal{O}_K$, it follows
from
Kılıçer-Labrande-Lercier-Ritzenthaler-Sijsling-Streng
\cite[Proposition 4.1]{KLLRSS} that $p$
has exactly $1$ or $3$ prime factors in~$\mathcal{O}_K$.
The number of such primes below this bound is roughly
$\frac{1}{16} B^{10} / \log(\frac{1}{8} B^{10}) $
by the prime number theorem and the Chebotarev density theorem.
The product of them will therefore have a number of digits
that is comparable to $B^{10}$ itself.

\vspace{2mm}


\textbf{Example 1.}
For the field $K_+=\Q(\alpha)=\Q[x]/(x^{3} - x^{2} - 2 x + 1)$,
let $K = K_+(\zeta_3)$. Then there 
is exactly one curve with primitive CM by the maximal order $\mathcal{O}_K$ of $K$.
A conjectural model is given in Koike-Weng~\cite[6.1(2)]{KoikeWeng} and Lario-Somoza~\cite[4.1.2]{LarioSomoza}.
Its invariants are
\begin{itemize}
\item  $i_1 = \frac{7^{4}}{2^{6}}$, $i_2 = \frac{-7^{2}}{2^{3}}$, $i_3 = \frac{-7^{3}}{2^{8}}$
\item  $\denDeltaabs{} = 2^{12} \approx 4.1\cdot 10^{3}$
, \quad $\frac{1}{8}B^{10} \approx 7.2\cdot 10^{10}$
\item  $j_1' = \frac{-2^{3}}{7^{2}}$, $j_2' = \frac{-1}{2^{2} \cdot 7}$
, \quad $\denKWabs{} = (2^{3} \cdot 7^{2})^{\frac{1}{2}} \approx 7.31861142004594$
\item  $j_1 = \frac{-7^{2}}{2^{3}}$, $j_2 = \frac{7}{2^{5}}$
, \quad $\denabs{} = 2^{3} \approx 8.00000000000000$
\end{itemize}
 
 \begin{itemize}
\item  $N_{-2 \alpha^{2} + 3} = (2^{28} \cdot 7 \cdot 13)^{3} \approx (2.4\cdot 10^{10})^{3}$
\end{itemize}

 \vspace{2mm} 
 
\textbf{Example 2.}
For the field $K_+=\Q(\alpha)=\Q[x]/(x^{3} - x^{2} - 4 x - 1)$,
let $K = K_+(\zeta_3)$. Then there 
is exactly one curve with primitive CM by $\mathcal{O}_K$.
A conjectural model is given in \cite[6.1(3)]{KoikeWeng} and \cite[4.1.3]{LarioSomoza}.
Its invariants are
\begin{itemize}
\item  $i_1 = \frac{(7^{6} \cdot 13)^{2}}{(2 \cdot 5)^{6}}$, $i_2 = \frac{-7^{6} \cdot 13 \cdot 47^{2}}{2^{3} \cdot 5^{4}}$, $i_3 = \frac{-7^{8} \cdot 13^{2} \cdot 31}{(2^{2} \cdot 5)^{4}}$
\item  $\denDeltaabs{} = (2^{2} \cdot 5)^{6} \approx 6.4\cdot 10^{7}$
, \quad $\frac{1}{8}B^{10} \approx 2.6\cdot 10^{13}$
\item  $j_1' = \frac{-2^{3} \cdot 5^{2} \cdot 47^{2}}{7^{6} \cdot 13}$, $j_2' = \frac{-5^{2} \cdot 31}{(2 \cdot 7^{2})^{2}}$
, \quad $\denKWabs{} = (2^{3} \cdot 7^{6} \cdot 13)^{\frac{1}{2}} \approx 2.3\cdot 10^{2}$
\item  $j_1 = \frac{-7^{6} \cdot 13}{2^{3} \cdot 5^{2} \cdot 47^{2}}$, $j_2 = \frac{7^{2} \cdot 13 \cdot 31}{2^{5} \cdot 47^{2}}$
, \quad $\denabs{} = 2^{3} \cdot 5 \cdot 47 \approx 1.9\cdot 10^{3}$
\end{itemize}
 
 \begin{itemize}
\item  $N_{-2 \alpha^{2} + 2 \alpha + 5} = (2^{51} \cdot 5^{6} \cdot 13 \cdot 31 \cdot 47)^{3} \approx (6.7\cdot 10^{23})^{3}$
\end{itemize}

 \vspace{2mm} 
 
\textbf{Example 3.}
For the field $K_+=\Q(\alpha)=\Q[x]/(x^{3} + x^{2} - 10 x - 8)$,
let $K = K_+(\zeta_3)$. Then there 
is exactly one curve with primitive CM by $\mathcal{O}_K$.
A conjectural model is given in \cite[6.1(4)]{KoikeWeng} and \cite[4.1.4]{LarioSomoza}.
Its invariants are
\begin{itemize}
\item  $i_1 = \frac{(7^{3} \cdot 31 \cdot 73^{3})^{2}}{(2^{3} \cdot 23)^{6}}$, $i_2 = \frac{-2 \cdot 7^{3} \cdot 31 \cdot 47^{2} \cdot 73^{3}}{23^{6}}$, $i_3 = \frac{-7^{5} \cdot 31^{2} \cdot 73^{4} \cdot 11593}{(2^{10} \cdot 23^{3})^{2}}$
\item  $\denDeltaabs{} = (2^{4} \cdot 23)^{6} \approx 2.5\cdot 10^{15}$
, \quad $\frac{1}{8}B^{10} \approx 1.2\cdot 10^{17}$
\item  $j_1' = \frac{-2^{19} \cdot 47^{2}}{7^{3} \cdot 31 \cdot 73^{3}}$, $j_2' = \frac{-11593}{2^{2} \cdot 7 \cdot 73^{2}}$
, \quad $\denKWabs{} = (2^{3} \cdot 7^{3} \cdot 31 \cdot 73^{3})^{\frac{1}{2}} \approx 3.2\cdot 10^{3}$
\item  $j_1 = \frac{-7^{3} \cdot 31 \cdot 73^{3}}{2^{19} \cdot 47^{2}}$, $j_2 = \frac{7^{2} \cdot 31 \cdot 73 \cdot 11593}{2^{21} \cdot 47^{2}}$
, \quad $\denabs{} = 2^{11} \cdot 47 \approx 9.6\cdot 10^{4}$
\end{itemize}
 
 \begin{itemize}
\item  $N_{-\alpha^{2} + \alpha + 7} = (2^{205} \cdot 23^{2} \cdot 29^{2} \cdot 31 \cdot 47^{2} \cdot 61^{2} \cdot 89 \cdot 101 \cdot 139)^{3} \approx (7.3\cdot 10^{81})^{3}$
\end{itemize}

 \vspace{2mm} 
 
\textbf{Example 4.}
For the field $K_+=\Q(\alpha)=\Q[x]/(x^{3} - x^{2} - 14 x - 8)$,
let $K = K_+(\zeta_3)$. Then there 
is exactly one curve with primitive CM by $\mathcal{O}_K$.
A conjectural model is given in \cite[6.1(5)]{KoikeWeng} and \cite[4.1.5]{LarioSomoza}.
Its invariants are
\begin{itemize}
\item  $i_1 = \frac{(7^{3} \cdot 43^{2} \cdot 223^{3})^{2}}{(2^{4} \cdot 11 \cdot 47)^{6}}$, $i_2 = \frac{-7^{3} \cdot 41^{2} \cdot 43^{2} \cdot 59^{2} \cdot 223^{3}}{2^{13} \cdot 11^{4} \cdot 47^{6}}$, $i_3 = \frac{-7^{4} \cdot 43^{3} \cdot 223^{4} \cdot 419 \cdot 431}{(2^{13} \cdot 11^{2} \cdot 47^{3})^{2}}$
\item  $\denDeltaabs{} = (2^{5} \cdot 11 \cdot 47)^{6} \approx 2.1\cdot 10^{25}$
, \quad $\frac{1}{8}B^{10} \approx 3.1\cdot 10^{18}$
\item  $j_1' = \frac{-2^{11} \cdot 11^{2} \cdot 41^{2} \cdot 59^{2}}{7^{3} \cdot 43^{2} \cdot 223^{3}}$, $j_2' = \frac{-11^{2} \cdot 419 \cdot 431}{2^{2} \cdot 7^{2} \cdot 43 \cdot 223^{2}}$
, \quad $\denKWabs{} = (2^{3} \cdot 7^{3} \cdot 43^{2} \cdot 223^{3})^{\frac{1}{2}} \approx 3.8\cdot 10^{4}$
\item  $j_1 = \frac{-7^{3} \cdot 43^{2} \cdot 223^{3}}{2^{11} \cdot 11^{2} \cdot 41^{2} \cdot 59^{2}}$, $j_2 = \frac{7 \cdot 43 \cdot 223 \cdot 419 \cdot 431}{2^{13} \cdot 41^{2} \cdot 59^{2}}$
, \quad $\denabs{} = 2^{7} \cdot 11 \cdot 41 \cdot 59 \approx 3.4\cdot 10^{6}$
\end{itemize}
 
 \begin{itemize}
\item  $N_{-2 \alpha + 1} = (2^{288} \cdot 11^{9} \cdot 41^{3} \cdot 43 \cdot 47^{2} \cdot 59^{3} \cdot 97 \cdot 131 \cdot 173 \cdot 211 \cdot 223 \cdot 269)^{3} \approx (4.4\cdot 10^{124})^{3}$
\end{itemize}

 \vspace{2mm} 
 
\textbf{Example 5.}
For the field $K_+=\Q(\alpha)=\Q[x]/(x^{3} - 21 x - 28)$,
let $K = K_+(\zeta_3)$. Then there 
is exactly one curve with primitive CM by $\mathcal{O}_K$.
A conjectural model is given in \cite[4.2.1.1]{LarioSomoza}.
Its invariants are
\begin{itemize}
\item  $i_1 = \frac{-3^{9} \cdot 5^{12} \cdot 7^{4}}{2^{18}}$, $i_2 = \frac{3^{3} \cdot 5^{6} \cdot 7^{2} \cdot 71^{2}}{2^{3}}$, $i_3 = \frac{3^{7} \cdot 5^{9} \cdot 7^{3} \cdot 2621}{2^{20}}$
\item  $\denDeltaabs{} = (2^{8} \cdot 3)^{3} \approx 4.5\cdot 10^{8}$
, \quad $\frac{1}{8}B^{10} \approx 2.1\cdot 10^{15}$
\item  $j_1' = \frac{-2^{15} \cdot 71^{2}}{(3^{3} \cdot 5^{3} \cdot 7)^{2}}$, $j_2' = \frac{-2621}{2^{2} \cdot 3^{2} \cdot 5^{3} \cdot 7}$
, \quad $\denKWabs{} = (2^{3} \cdot 3^{6} \cdot 5^{6} \cdot 7^{2})^{\frac{1}{2}} \approx 1.6\cdot 10^{3}$
\item  $j_1 = \frac{-(3^{3} \cdot 5^{3} \cdot 7)^{2}}{2^{15} \cdot 71^{2}}$, $j_2 = \frac{3^{4} \cdot 5^{3} \cdot 7 \cdot 2621}{2^{17} \cdot 71^{2}}$
, \quad $\denabs{} = 2^{9} \cdot 71 \approx 3.6\cdot 10^{4}$
\end{itemize}
 
 \begin{itemize}
\item  $N_{2 \alpha} = 2^{433} \cdot 3^{55} \cdot 7^{11} \cdot 31^{3} \cdot 47^{3} \cdot 59^{3} \cdot 61^{3} \cdot 71^{3} \cdot 173^{3} \approx (1.3\cdot 10^{66})^{3}$
\end{itemize}

 \vspace{2mm} 
 
\textbf{Example 6.}
For the field $K_+=\Q(\alpha)=\Q[x]/(x^{3} - 21 x - 35)$,
let $K = K_+(\zeta_3)$. Then there 
is exactly one curve with primitive CM by $\mathcal{O}_K$.
A conjectural model is given in \cite[4.2.1.2]{LarioSomoza}.
Its invariants are
\begin{itemize}
\item  $i_1 = \frac{-2^{12} \cdot 3^{9} \cdot 7^{4} \cdot 37^{6}}{(5 \cdot 11 \cdot 23)^{6}}$, $i_2 = \frac{2^{6} \cdot 3^{3} \cdot 7^{2} \cdot 37^{3} \cdot 149^{2} \cdot 257^{2}}{(5^{2} \cdot 11^{3} \cdot 23^{3})^{2}}$, $i_3 = \frac{2^{9} \cdot 3^{7} \cdot 7^{3} \cdot 37^{4} \cdot 2683}{(5^{2} \cdot 11^{3} \cdot 23^{3})^{2}}$
\item  $\denDeltaabs{} = (3 \cdot 5^{2} \cdot 11^{2} \cdot 23^{2})^{3} \approx 1.1\cdot 10^{20}$
, \quad $\frac{1}{8}B^{10} \approx 2.1\cdot 10^{15}$
\item  $j_1' = \frac{-(5 \cdot 149 \cdot 257)^{2}}{2^{6} \cdot 3^{6} \cdot 7^{2} \cdot 37^{3}}$, $j_2' = \frac{-5^{2} \cdot 2683}{2^{3} \cdot 3^{2} \cdot 7 \cdot 37^{2}}$
, \quad $\denKWabs{} = (2^{6} \cdot 3^{6} \cdot 7^{2} \cdot 37^{3})^{\frac{1}{2}} \approx 4.9\cdot 10^{3}$
\item  $j_1 = \frac{-2^{6} \cdot 3^{6} \cdot 7^{2} \cdot 37^{3}}{(5 \cdot 149 \cdot 257)^{2}}$, $j_2 = \frac{2^{3} \cdot 3^{4} \cdot 7 \cdot 37 \cdot 2683}{(149 \cdot 257)^{2}}$
, \quad $\denabs{} = 5 \cdot 149 \cdot 257 \approx 1.9\cdot 10^{5}$
\end{itemize}
 
 \begin{itemize}
\item  $N_{-2 \alpha^{2} + 4 \alpha + 28} = 2^{245} \cdot 3^{58} \cdot 5^{36} \cdot 7^{8} \cdot 11^{12} \cdot 23^{6} \cdot 71^{3} \cdot 149^{3} \cdot 257^{3} \approx (5.9\cdot 10^{57})^{3}$
\end{itemize}

 \vspace{2mm} 
 
\textbf{Example 7.}
For the field $K_+=\Q(\alpha)=\Q[x]/(x^{3} - 39 x - 26)$,
let $K = K_+(\zeta_3)$. Then there 
is exactly one curve with primitive CM by $\mathcal{O}_K$.
A conjectural model is given in \cite[4.2.1.3]{LarioSomoza}.
Its invariants are
\begin{itemize}
\item  $i_1 = \frac{-3^{9} \cdot 5^{12} \cdot 7^{6} \cdot 11^{6} \cdot 13^{2}}{(2^{5} \cdot 29)^{6}}$, $i_2 = \frac{3^{3} \cdot 5^{6} \cdot 7^{3} \cdot 11^{5} \cdot 13 \cdot 59^{2} \cdot 149^{2}}{2^{19} \cdot 29^{6}}$, $i_3 = \frac{3^{7} \cdot 5^{9} \cdot 7^{5} \cdot 11^{4} \cdot 13^{2} \cdot 17 \cdot 17669}{(2^{16} \cdot 29^{3})^{2}}$
\item  $\denDeltaabs{} = (2^{12} \cdot 3 \cdot 29^{2})^{3} \approx 1.1\cdot 10^{21}$
, \quad $\frac{1}{8}B^{10} \approx 1.\cdot 10^{18}$
\item  $j_1' = \frac{-2^{11} \cdot 59^{2} \cdot 149^{2}}{3^{6} \cdot 5^{6} \cdot 7^{3} \cdot 11 \cdot 13}$, $j_2' = \frac{-17 \cdot 17669}{2^{2} \cdot 3^{2} \cdot 5^{3} \cdot 7 \cdot 11^{2}}$
, \quad $\denKWabs{} = (2^{3} \cdot 3^{6} \cdot 5^{6} \cdot 7^{3} \cdot 11^{3} \cdot 13)^{\frac{1}{2}} \approx 8.1\cdot 10^{4}$
\item  $j_1 = \frac{-3^{6} \cdot 5^{6} \cdot 7^{3} \cdot 11 \cdot 13}{2^{11} \cdot 59^{2} \cdot 149^{2}}$, $j_2 = \frac{3^{4} \cdot 5^{3} \cdot 7^{2} \cdot 13 \cdot 17 \cdot 17669}{2^{13} \cdot 11 \cdot 59^{2} \cdot 149^{2}}$
, \quad $\denabs{} = 2^{7} \cdot 11 \cdot 59 \cdot 149 \approx 1.2\cdot 10^{7}$
\end{itemize}
 
 \begin{itemize}
\item  $N_{-\frac{1}{2} \alpha^{2} + \frac{3}{2} \alpha + 13} = 2^{921} \cdot 3^{100} \cdot 11^{21} \cdot 13^{8} \cdot 29^{6} \cdot 53^{3} \cdot 59^{6} \cdot 109^{3} \cdot 113^{3} \cdot 149^{6} \cdot 233^{3} \cdot 359^{3} \cdot 467^{3} \cdot 541^{3} \cdot 577^{3} \approx (2.\cdot 10^{148})^{3}$
\end{itemize}

 \vspace{2mm} 
 
\textbf{Example 8.}
For the field $K_+=\Q(\alpha)=\Q[x]/(x^{3} - 61 x - 183)$,
let $K = K_+(\zeta_3)$. Then there 
are exactly four curves with primitive CM by $\mathcal{O}_K$.
Conjectural models are given in \cite[4.3.1 (corrected w.r.t.~arXiv version 1)]{LarioSomoza}.
Let $H_{j_1}=H_{\mathcal{O}_K,1}$ and $\widehat{H}_{j_1,j_2}=\widehat{H}_{\mathcal{O}_K, 2}$ be 
the polynomials as in \eqref{eq:classpol1}--\eqref{eq:classpol2}, 
and let ${H}_{j_1'}$ and $H_{i_1}$ be defined as in \eqref{eq:classpol1}, 
but with the invariants $j_1'$ and $i_1$ instead of~$j_1$.
 These polynomials are numerically approximable with the methods of Koike-Weng~\cite{KoikeWeng} and Lario-Somoza~\cite{LarioSomoza} 
 and satisfy $H_{j_1}(j_1(C))=0$ and $j_2(C) = \widehat{H}_{j_1,j_2}(j_1(C)) / H_{j_1}'(j_1(C))$ 
(see~\cite{GHKRW}). 
Then the denominators of these polynomials are\begin{itemize}
\item $\mathrm{den}(H_{j_1}) = 2^{3} \cdot 3^{39} \cdot 11^{2} \cdot 23^{2} \cdot 41^{2} \cdot 53^{4} \cdot 89^{2} \cdot 113^{2} \cdot 149^{2} \cdot 191^{2}\approx 2.3\cdot 10^{51}$
\\ $\mathrm{den}(\widehat{H}_{j_1,j_2}) = 2^{5} \cdot 3^{39} \cdot 11^{3} \cdot 23^{2} \cdot 41^{2} \cdot 53^{4} \cdot 89^{2} \cdot 113^{2} \cdot 149^{2} \cdot 191^{2}\approx 9.9\cdot 10^{52}$
\item $\mathrm{den}(H_{j_1'}) = 2^{18} \cdot 7^{12} \cdot 11 \cdot 61^{2} \cdot 1289^{3} \cdot 6551^{3} \cdot 20707^{3}\approx 7.9\cdot 10^{53}$
\item $\mathrm{den}(H_{i_1}) = (2^{2} \cdot 3^{9} \cdot 11^{6} \cdot 23^{4} \cdot 53^{2} \cdot 131^{2})^{3}\approx 6.7\cdot 10^{72}$
\end{itemize}
 
 \begin{itemize}
\item  $N_{-4 \alpha^{2} + 18 \alpha + 163} = (2^{235} \cdot 3^{148} \cdot 11^{12} \cdot 23^{6} \cdot 37^{3} \cdot 41^{3} \cdot 53^{3} \cdot 61 \cdot 89 \cdot 113 \cdot 131 \cdot 149 \cdot 191 \cdot 367 \cdot 613 \cdot 643 \cdot 733 \cdot 907)^{3} \approx (1.2\cdot 10^{203})^{3}$
\end{itemize}

 \vspace{2mm} 
 
\textbf{Example 9.}
For the field $K_+=\Q(\alpha)=\Q[x]/(x^{3} - x^{2} - 22 x - 5)$,
let $K = K_+(\zeta_3)$. Then there 
are exactly four curves with primitive CM by $\mathcal{O}_K$.
Conjectural models are given in \cite[4.3.2]{LarioSomoza}.
In the notation of Example~8, we have \begin{itemize}
\item $\mathrm{den}(H_{j_1}) = (2^{6} \cdot 3^{15} \cdot 5^{4} \cdot 89 \cdot 137 \cdot 149 \cdot 179 \cdot 269)^{2}\approx 2.5\cdot 10^{45}$
\\ $\mathrm{den}(\widehat{H}_{j_1,j_2}) = (2^{6} \cdot 3^{15} \cdot 5^{3} \cdot 89 \cdot 137 \cdot 149 \cdot 179 \cdot 269)^{2}\approx 1.\cdot 10^{44}$
\item $\mathrm{den}(H_{j_1'}) = 7^{12} \cdot 53^{3} \cdot 67 \cdot 107^{3} \cdot 179^{3} \cdot 3029017^{3}\approx 2.7\cdot 10^{49}$
\item $\mathrm{den}(H_{i_1}) = (2^{4} \cdot 3^{5} \cdot 5^{4} \cdot 53 \cdot 59 \cdot 107)^{6}\approx 2.9\cdot 10^{71}$
\end{itemize}
 
 \begin{itemize}
\item  $N_{-\frac{2}{3} \alpha^{2} + \frac{29}{3}} = (2^{253} \cdot 3^{155} \cdot 5^{32} \cdot 43^{2} \cdot 53^{3} \cdot 59^{2} \cdot 67 \cdot 89^{2} \cdot 107 \cdot 109 \cdot 137 \cdot 149 \cdot 179^{2} \cdot 223 \cdot 241 \cdot 263 \cdot 269 \cdot 397 \cdot 643 \cdot 997 \cdot 1087)^{3} \approx (1.2\cdot 10^{224})^{3}$
\end{itemize}

 \vspace{2mm}

\vspace{2mm}    

\begin{remark}
	Notice that the size of the denominators of the absolute invariants $j_1$ and $j_2$
	of Section~\ref{sec:invariants}
	is similar to the denominators of the Koike-Weng invariants and much 
	smaller that the denominators of the invariants defined by using the discriminant.
	Theorem $1.3$ in \cite{KLLNOS} suggests a bound for the primes appearing in the discriminant,
	while we do not have a bound at all for the primes in the denominator of the Koike-Weng invariants.
	That the primes in the denominators of the Koike-Weng invariants are small,
	and even smaller than
	those for our invariants, is a mystery that needs further research.
	 
	For our absolute invariants, we have the best bound. 
	Hence among the three kind of invariants the more suitable ones for constructing Picard 
	curves with CM by a given order $\mathcal{O}$ are the absolute invariants $j_1$ and~$j_2$.
\end{remark}
	
\appendix

	\section{A lemma about components of bad reduction}
	\label{appendix}
	
	Most of this appendix is an edited copy of an email from Bas Edixhoven
	to the authors. Lemma~\ref{lem:components_of_bad_red}
	below and its proof are well-known to many experts,
	but it seems that neither
	 is written down in the literature.
	For completeness, as we use it in our proof of the singular case of Proposition~\ref{prop:decomp},
	we state the lemma
	and provide details of the proof.
	
	\begin{lemma} \label{lem:components_of_bad_red}
		
		Let $R$ be a discrete valuation ring with fraction field $M$
		and residue field~$k$.
		Let $X$ be a
		projective $R$-scheme, of dimension $2$, flat over $R$ 
		such that $X_M$ is smooth and geometrically connected over~$M$, 
		of genus at least~$1$. Let $C$ be an
		irreducible component of $X_k$, and assume that $C$ is geometrically irreducible
		and birational to a smooth geometrically irreducible
		projective curve $C'$ over $k$ of genus at least~$1$.
		
		Suppose that there exists an open subscheme $U$ of $X$ 
		such that $U$ is smooth over $R$, and such that $U_k$ is a non-empty 
		open subset of~$C$.
		
		Let $M \rightarrow M'$ be a finite separable field 
		extension such that $X_M$ 
		has stable reduction over the integral closure $R^{'}$ of 
		$R$ in $M^{'}$.
		Then the open subscheme $U_{R'}$ of $X_{R'}$ is isomorphic 
		to an open subscheme of the stable model of~$X$.
		Moreover, 
		the normalization of the special fibre of 
		the stable model of $X$ contains a copy of~$C^{'}_{k'}$,
		where $k'$ is the residue field of~$R'$.
	\end{lemma}

	\begin{proof}
		Let $X_{R'}$ be the pullback of $X$ via 
		$R \rightarrow R'$, which is integral by Proposition 4.3.8 of Liu~\cite{Liu},
		and let
		$X_{R'}^\text{stab}$ be the (unique) stable model of $X_{M'}$ over $R'$.
		
		We apply Corollary 8.3.51 of Liu~\cite{Liu} to $X_{R'}$
		(and see \cite[Definition 8.3.39]{Liu} to see what `in the strong sense' means).
		This gives us $f: X_{R'}^\text{res} \rightarrow X_{R'}^{\phantom{\text{r}}}$ birational, with 
		$X_{R'}^\text{res}$ projective over~$R'$
		and regular, with 
		$f$ an isomorphism over the open 
		subscheme $U_{R'}$ of $X_{R'}$. Here we use that $U$ is smooth over $R$, 
		hence $U_{R'}$ is smooth over $R'$, hence $U_{R'}$ is regular.
		
		Theorem 9.3.21 in \cite{Liu} says that there is a unique minimal regular
		model $X_{R'}^\text{res}\rightarrow X_{R'}^\text{min}$ of $X_{R'}^\text{res}$,
		which is in fact
		(see the proof)
		isomorphic to every relatively minimal model.
		This morphism is the identity on the generic fibres and
		by the construction
		(Castelnuovo's criterion (Theorem 9.3.8) and Proposition 9.3.19 in \cite{Liu})
		contracts precisely the irreducible components $E$ of the closed fibre 
		of
		$X_{R'}^\text{res}$ such that $E$ is isomorphic to $\PP^1_{k_E}$ with $k_E :=
		H^0(E,O_E)$ (a finite extension of the residue field $k'$ of $R'$ over
		which $E$ lies), and that $E\cdot E = -[k_E:k']$.
		
		Note that the open subscheme $U_{R'}$ of $X_{R'}^\text{res}$ is mapped
		isomorphically to an open subscheme of $X_{R'}^\text{min}$, because 
		its closed
		fibre is an open part of a curve of genus $\geq 1$.
		
		Corollary 10.3.25 of \cite{Liu} says that there is a 
		unique morphism to $X_{R'}^\text{stab}$ from its minimal desingularization
		$(X_{R'}^\text{stab})^{\mathrm{mindes}}$
		and that this morphism only contracts $\PP^1$'s in the closed
		fibres of self-intersection~$-2$.
		As the geometric special fibre of
		$X_{R'}^\text{stab}$ has no $\PP^1$'s, except with self-intersection $\leq -3$
		(Definitions 10.3.1--2 of \cite{Liu}),
		we get that the geometric special fibre of
		$(X_{R'}^\text{stab})^{\mathrm{mindes}}$ has
		no $\PP^1$'s except with self-intersection $\leq -2$.

		Exactly like $X_{R'}^\text{res}$, the surface $(X_{R'}^\text{stab})^{\text{mindes}}$
		also has a morphism to $X_{R'}^\text{min}$ that only contracts
		curves that are (after field extension)
		$\PP^1$'s of self-intersection $-1$,
		and as $(X_{R'}^\text{stab})^{\text{mindes}}$ has no such $\PP^1$'s,
		this morphism is an isomorphism.
	
\begin{figure}
	\[
\xymatrix{& (X_{R'}^\mathrm{stab})^{\rlap{{\scriptsize mindes}}} 
	\ar@{=}[dr]
	\ar[dl]& & X_{R'}^{\mathrm{res}} \ar[dr]\ar[dl] \\
X_{R'}^{\mathrm{stab}} & &  X_{R'}^{\mathrm{min}} & & X_{R'}}
\]
\caption{The fibred surfaces in the proof of Lemma~\ref{lem:components_of_bad_red}.}
\label{fig:zigzag}
\end{figure}
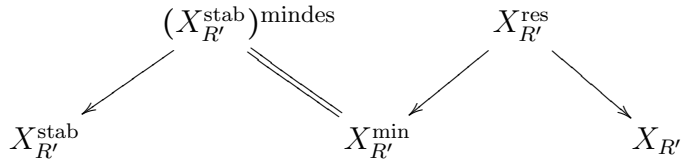

		Therefore, through the maps
		of Figure~\ref{fig:zigzag}, the open subscheme
		$U_{R'}\subset X_{R'}$ is mapped isomorphically to an open subscheme of
		$X_{R'}^\text{stab}$.
		
		When we base change $U_{R'}$ to the residue field $k'$,
		we get an embedding 
		$U_{k'}\rightarrow X_{k'}^{\mathrm{stab}}$, 
		where~$U_{k'}$ is a base change of a non-empty open smooth
		part of~$C$, hence a base change of a dense part of~$C'$.
		Let $C^{''}$ be the closure of the image of this embedding
		of curves. Then we get a birational map of
		curves $C^{'}_{k'}\rightarrow C^{''}$,
		hence an isomorphism to the normalization.
	\end{proof}
	
	\begin{cor}\label{cor:components_of_bad_red}
	    Let $R$ be a discrete valuation ring with maximal ideal $\mathfrak{m}$,
	    field of fractions~$M$, and residue field $k=R/\mathfrak{m}$
	    of characteristic not $2$ or~$3$.
	    Let $D$ be a smooth projective, geometrically irreducible
	    curve over $M$ and let suppose that $R$ is such that $D$ has
	    stable reduction over~$R$.
	    \begin{enumerate}
	    	\item If $D$ over $M$ is given by
	    	$y^3 = x^4 + ax^2 + bx + 1$
	    	with $b, a \pm 2\in \mathfrak{m}$,
	    	then the stable reduction 
	    	$\overline{D}$ of $D$ has an irreducible component birational to the 
	    	elliptic curve $C' : Y^2 = X^3 \pm 1$,
	    	\item If $D$ over $M$ is given by $y^3 = x^4 + x^2 + bx + c$
	    	with $b, c\in \mathfrak{m}$, then the stable reduction 
			$\overline{D}$ of $D$ has an irreducible component birational to the 
			hyperelliptic curve $C' : Y^2 = X^6-4$.
	    \end{enumerate}
	\end{cor}
\begin{proof}
	Let $X$ (respectively $C$) be the plane projective
	$R$-scheme (respectively $k$-scheme)
	given by the defining polynomial $f$ of~$D$.
	Let $U=\mathrm{Spec}(R[x,y,y^{-1}]/(f))$.
	%
	%
	By Lemma~\ref{lem:components_of_bad_red}, it now
	suffices to check that $C$ is birational to $C'$,
	which we did in the proof of Lemma~\ref{lem:cases}.
\end{proof}

\bibliographystyle{plain}
\bibliography{biblio}
\Addresses

\end{document}